\documentclass[journal]{IEEEtran}

\usepackage{amsmath,amsfonts,amsthm,amssymb,mathrsfs}
\usepackage{enumitem}
\usepackage{array}
\usepackage[caption=false,font=normalsize,labelfont=sf,textfont=sf]{subfig}
\usepackage{textcomp}
\usepackage{stfloats}
\usepackage{url}
\usepackage{verbatim}
\usepackage{graphicx}
\usepackage{cite}
\usepackage{cancel}

\usepackage{algorithm,setspace,soul}
\usepackage{algpseudocode}
\usepackage{breqn}
\usepackage{nicefrac}
\usepackage{bbm,xspace}

\usepackage{balance}

\usepackage{tikz}
\usetikzlibrary{fadings}
\usetikzlibrary{patterns}
\usetikzlibrary{shadows.blur}
\usetikzlibrary{shapes}

\usepackage{booktabs}
\usepackage{multirow}

\newcommand{\etal}			{{et al.\ }}
\newcommand{\bpara}[1]		{\smallskip \noindent {\bf #1}}

\usepackage[scaled=0.87]{cascadia-code}

\def\ompfbp				{{{\texttt{OMP-FBP}}}\xspace}
\def\ompdfr			    {{{\texttt{OMP-DFR}}}\xspace}
\def\ompnfft			{{{\texttt{OMP-NFFT}}}\xspace}
\def\usfbp  			{{{\texttt{US-FBP}}}\xspace}
\def\madc			    {{$\mathscr{M}_\lambda$--{\texttt{ADC}}}\xspace}

\newcommand{\secref}[1]		{Section~\ref{#1}}
\newcommand{\fig}[1]		{Fig.~\ref{#1}}

\DeclareFontFamily{U}{mathx}{}
\DeclareFontShape{U}{mathx}{m}{n}{<-> mathx10}{}
\DeclareSymbolFont{mathx}{U}{mathx}{m}{n}
\DeclareMathAccent{\widehat}{0}{mathx}{"70}
\DeclareMathAccent{\widecheck}{0}{mathx}{"71}

\newtheorem{definition}{Definition}
\newtheorem{proposition}{Proposition}
\newtheorem{theorem}{Theorem}
\newtheorem{corollary}{Corollary}
\newtheorem{lemma}{Lemma}

\algrenewcommand\algorithmicrequire{\textbf{Input:}}
\algrenewcommand\algorithmicensure{\textbf{Output:}}

\def\N			{\mathbb N}
\def\Z			{\mathbb Z}
\def\R			{\mathbb R}
\def\C			{\mathbb C}
\def\Sphere		{\mathbb{S}^1}

\def\Radon		{\mathcal R}

\def\Fourier	{\mathcal F}

\def\Mod		{\mathscr M}
\def\Modulo		{\mathscr R}
\def\Sum		{\mathsf S}
\def\ind		{\mathbbm 1}

\def\Lebesgue	{\mathrm L}

\def\Band		{\mathcal B}

\def\d			{\mathrm d}
\def\e			{\mathrm e}
\def\i			{\mathrm i}
\def\T			{\mathrm T}

\def\bfx		{\mathbf{x}}

\def\bfz		{\mathbf{z}}
\def\bftheta	{\boldsymbol{\theta}}
\def\bfomega	{\boldsymbol{\omega}}
\def\bfV		{\mathbf{V}}
\def\bfc		{\mathbf{c}}
\def\bfs		{\mathbf{s}}

\def\Ell		{\mathbb L}
\def\E			{\mathbb E}
\def\M			{\mathcal M}
\def\S			{\mathcal S}

\def\O			{\mathcal O}

\def\OF			{\mathsf{OF}}
\def\complement	{\mathsf c}

\def\FBP		{\text{FBP}}
\def\DFR		{\text{DFR}}
\def\NFFT		{\text{NFFT}}

\def\dB			{\mathrm{dB}}

\DeclareMathOperator{\SSIM}{SSIM}

\DeclareMathOperator{\supp}{supp}

\DeclareMathOperator*{\argmax}{arg\,max}
\DeclareMathOperator*{\argmin}{arg\,min}

\begin{document}

\title{Fourier-Domain Inversion for the \\ Modulo Radon Transform}

\author{Matthias Beckmann, Ayush Bhandari, and Meira Iske%
\thanks{Initial results pertaining to this work were presented at IEEE ICIP 2022~\cite{Beckmann2022a}.}%
\thanks{This work was supported by the U.K.~Research and Innovation
council’s Future Leaders Fellowship program ``Sensing Beyond Barriers'' (MRC Fellowship MR/S034897/1) and the Deutsche Forschungsgemeinschaft (DFG) - Project number 530863002.}%
\thanks{M.~Beckmann is with the Center for Industrial Mathematics, University of Bremen, Germany and also with the Dept.\ of Electrical and Electronic Engineering, Imperial College London, U.K.~(Email: \texttt{research@mbeckmann.de}).}%
\thanks{A.~Bhandari is with the Dept.\ of Electrical and Electronic
Engineering, Imperial College London, U.K.~(Email:\texttt{ ayush@alum.MIT.edu}).}%
\thanks{M.~Iske is with the Center for Industrial Mathematics, University of Bremen, Germany (Email:\texttt{iskem@uni-bremen.de}) and acknowledges funding by the Deutsche Forschungsgemeinschaft (DFG) - Project number 281474342.}%
\thanks{The authors are listed in alphabetical order.}}

\markboth{Beckmann \MakeLowercase{\textit{et al.}}: Fourier-Domain Inversion for the Modulo Radon Transform}%
{Beckmann \MakeLowercase{\textit{et al.}}: Fourier-Domain Inversion for the Modulo Radon Transform}

\maketitle
\thispagestyle{empty}

\begin{abstract}
Inspired by the multiple-exposure fusion approach in computational photography, recently, several practitioners have explored the idea of high dynamic range (HDR) X-ray imaging and tomography. While establishing promising results, these approaches inherit the limitations of multiple-exposure fusion strategy. To overcome these disadvantages, the modulo Radon transform (MRT) has been proposed. The MRT is based on a co-design of hardware and algorithms. In the hardware step, Radon transform projections are folded using modulo non-linearities. Thereon, recovery is performed by algorithmically inverting the folding, thus enabling a single-shot, HDR approach to tomography. The first steps in this topic established rigorous mathematical treatment to the problem of reconstruction from folded projections. This paper takes a step forward by proposing a new, Fourier domain recovery algorithm that is backed by mathematical guarantees. The advantages include recovery at lower sampling rates while being agnostic to modulo threshold, lower computational complexity and empirical robustness to system noise. Beyond numerical simulations, we use prototype modulo ADC based hardware experiments to validate our claims. In particular, we report image recovery based on hardware measurements up to 10 times larger than the sensor's  dynamic range while benefiting with lower quantization noise ($\sim$12 dB).
\end{abstract}

\begin{IEEEkeywords}
X-ray computerized tomography, high dynamic range, Radon transform, modulo non-linearity, sampling theory.
\end{IEEEkeywords}

\section{Introduction}

In the recent years, practitioners in the area of tomography have started to develop methods for high dynamic range (HDR) imaging. HDR tomography allows for recovery of tomograms beyond the restrictions imposed by the detector's fixed albeit limited dynamic range. The first approaches were inspired by computational photography for HDR imaging \cite{Debevec1997,Bhandari:2022:B}. Such approaches rely on \emph{multi-exposure fusion}. In the tomography context, this translates to acquiring multiple low dynamic range (LDR) measurements at a different energy/exposure level, thus accounting for multi-exposures \cite{Bhandari:2022:B}. Subsequently, the LDR images are algorithmically ``fused'' into a single, HDR image. In summary, given a detector of fixed dynamic range, this strategy enables recovery of tomographic images with a dynamic range that is much larger than the detector's maximum range.

\begin{figure}[!t]
\centering
\includegraphics[width=0.95\linewidth]{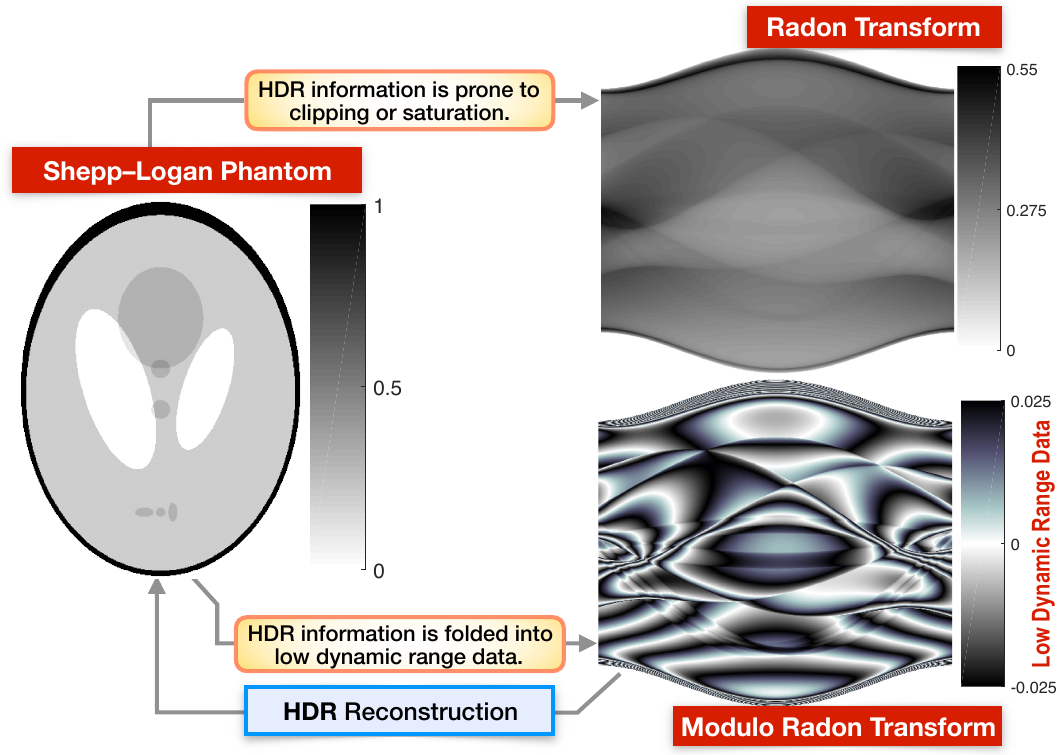}
\caption{Overview of the computational imaging pipeline for single-shot HDR tomography via the modulo Radon transform (MRT) \cite{Bhandari2020,Beckmann2020,Beckmann2022}.}
\label{fig:scheme}
\end{figure}

\bpara{Emerging Literature on HDR Tomography.} 
First attempts by Trpovski \etal \cite{Trpovski2013} on HDR recovery were based on X-ray image pair fusion.
Follow-up works by Chen \etal \cite{Chen2015} and by Haidekker \etal \cite{Haidekker2017} were based on a multiple exposure approach for X-ray imaging. For an illustrative example of HDR X-ray, we refer the readers to Fig.~6 in \cite{Haidekker2017}. A pixel-level design for an HDR X-ray imaging setup was considered by Weiss \etal \cite{Weiss2017}. Akin to computational photography, multi-exposure X-ray imaging requires acquisition with different gains. This is achieved by  careful calibration at each exposure. The work of Li \etal \cite{Li2018} presents an automated approach to this problem. Going beyond the case of single projection angle, exposure adaption can be performed across the scanning angles. This approach was investigated by Chen \etal in \cite{Chen2020}.

\bpara{HDR Tomography via Modulo Radon Transform.} On the one hand, the approaches in \cite{Chen2015,Haidekker2017,Weiss2017,Li2018,Chen2020} take a new step forward in terms of HDR X-ray imaging. On the other hand, since the basic technique is pivoted on HDR photography \cite{Debevec1997}, in the tomography context, the challenges posed by the limitations intrinsic to multi-exposure fusion still need to be resolved. These limitations include, (i) ghosting artifacts, (ii) exposure calibration problem, (iii) efficient tone mapping methods, and, (iv) unknown sensor response. Detailed aspects of these limitations are discussed in \cite{Beckmann2022}. Further to these limitations attributed to multi-exposure fusion, we note that prior work on HDR X-ray imaging \cite{Chen2015,Haidekker2017,Weiss2017,Li2018,Chen2020} is focused on empirical experiments. This necessitates the development of principled approaches for HDR X-ray imaging that are backed by mathematical guarantees and lead to efficient reconstruction algorithms amenable to practical scenarios. 

Towards this end, in the recent line of works, the modulo Radon transform (MRT) \cite{Bhandari2020,Beckmann2020,Beckmann2022} was introduced to facilitate a single-shot, HDR tomography solution. MRT based HDR imaging jointly harnesses a collaboration between hardware and algorithms. A breakdown of the pipeline is as follows:
\begin{enumerate}[label = $\bullet$,leftmargin = *]
\item {\bf Hardware.} In the hardware pipeline, instead of digitizing the conventional Radon transform (RT), in the MRT case, one ``folds'' the RT in analog or continuous-time domain via modulo non-linearities. This maps the HDR, RT projection into LDR, folded signal which is subsequently sampled or digitized using analog-to-digital converters. The folding structure is inspired by the Unlimited Sensing Framework (USF) \cite{Bhandari:2017:C,Bhandari:2020:Ja,Bhandari:2021:J}. Injecting modulo non-linearity in the continuous-time domain has two distinct advantages. 
  
\begin{enumerate}[leftmargin = *]

\item Conventionally, due to their arbitrary dynamic range, the RT projections may saturate the detector \cite{Chen2015}. This results in permanent loss of information via saturated or clipped measurements. In our case, the folding approach ensures that the MRT measurements are bounded by the modulo threshold \cite{Beckmann2022,Bhandari:2021:J}. 

\item Given a fixed bit budget or digital resolution, lower dynamic range results in higher quantization resolution. Since MRT measurements are bounded by a much smaller range than the RT projections, the resulting measurements leverage the advantage of lower quantization noise floor.
\end{enumerate} 
  
\item {\bf Algorithms.} The inversion of MRT leads to a new class of non-linear inverse problems. Recovery algorithms that ``unfold'' LDR projections into HDR measurements constitute the reconstruction pipeline for MRT recovery. For first results on this topic, we refer to \cite{Beckmann2022}.

\end{enumerate}

\bpara{Existing Art.} In our recent works \cite{Bhandari2020,Beckmann2020,Beckmann2022}, we have considered the inversion of the MRT using a sequential approach. Starting with folded RT projections, these approaches work by first ``unfolding'' the MRT projections per scan angle and then using existing methods for inverting the RT, e.g. filtered back projection. Rigorous treatment of the topic together with a proof-of-concept hardware validation was demonstrated in \cite{Beckmann2022}. These ideas provide an \emph{end-to-end} implementation of the MRT approach, clearly showing its potential value in terms of HDR tomography. Simultaneously, since computational sensing and imaging methods rely on a melding of hardware and algorithms, this work also raises several interesting questions on the theoretical and practical frontiers.

\bpara{Motivation and Contributions.}
The end-to-end implementation of the MRT, tying both the theory and practice aspects in a holistic fashion, motivate the following observations. In the context of our previous recovery approach \cite{Bhandari2020,Beckmann2020,Beckmann2022},
\begin{enumerate}[label = $\bullet$,leftmargin = *]
\item reconstruction is based on the inversion of higher-order differences. This leads to numerical instabilities specially in the presence of noise, thus limiting the true potential of HDR recovery. This calls for development of efficient and stable recovery methods. 

\item a factor of $\pi\e$ oversampling is required. How can we achieve tighter sampling rates? This will ease the burden on reconstruction algorithms.

\item modulo ADC calibration is a must as the knowledge of ADC threshold $\lambda$ is required for the algorithm to work. Hence, an algorithm that is agnostic or blind to this parameter is highly desirable. 
\end{enumerate}

Continuing along the directions of our recent presentation \cite{Beckmann2022a}, in this paper, we address the above points by developing a Fourier domain recovery method. Note that the existing methods \cite{Bhandari2020,Beckmann2020,Beckmann2022} are solely based on spatial domain processing. In that regard, our Fourier domain algorithm broadens the algorithmic scope when it comes to using the full range of recovery methods linked with tomography. A major challenge in deriving Fourier domain perspective in our setting is that working with non-linearities is particularly difficult in the transform domain. To this end, our work takes a step forward in developing new insights for inversion. In summary, our main contributions are as follows: 
\begin{enumerate}[label = $\mathrm{C}_{\arabic*})$,leftmargin = *]
\item {\bf Algorithmic Novelty.} We present a novel, non-sequential algorithm for inversion of the MRT that directly works in the Fourier domain. The advantages are twofold. 

\begin{enumerate}[leftmargin = *]
\item It leads to an efficient algorithmic implementation. Firstly, the runtime \textbf{algorithmic complexity} analysis is presented in Section~\ref{subsec:complexity} and shows a reduction of computational costs by one order of magnitude (up to logarithmic factor). Secondly, working directly in the Fourier domain enables backwards compatibility with existing tools for Fourier RT reconstruction e.g. the methods surveyed in Section~\ref{subsec:DFR_Radon}.

\item Our algorithm is agnostic to the folding threshold $\lambda$ which circumvents the problem of ADC calibration.
\end{enumerate}

\item {\bf Recovery Guarantees.} Our algorithm is backed by mathematical guarantees presented in Section~\ref{subsec:rec_guarantee}. In particular, our recovery algorithm works with any sampling rate above Nyquist rate (see Corollary~\ref{cor:rec_setting}). This is factor $\pi \e$ improvement over previous result. 

\item {\bf Hardware Validation.} Experiments using modulo ADC hardware \cite{Bhandari:2021:J} shows the distinct advantages of our underpinning theory. In particular, (a) recovery at much lower sampling rates, (b) HDR recovery with higher digital resolution or lower quantization noise, and, (c) empirical robustness in the presence of system noise and outliers.
\end{enumerate}

\bpara{Paper Overview.}
This paper is organized as follows. We begin our discussion by revisiting the modulo Radon transform and the associated forward model in \secref{sec:MRT}. Our main results are covered in \secref{sec:theory}. In \secref{subsec:MRTOMP} we present a sequential reconstruction approach; the associated theoretical guarantees are derived in \secref{subsec:rec_guarantee}. A direct reconstruction approach is presented in \secref{subsec:DFR} which makes our approach compatible with a variety of exiting methods for the conventional Radon transform briefly reviewed in \secref{subsec:DFR_Radon}. Computational complexity analysis is presented in \secref{subsec:complexity}. 
\secref{sec:numerics} is dedicated to numerical and hardware experiments which validate the methods developed in \secref{sec:theory}.
Finally, we summarize our results and point to future research in \secref{sec:conc}.

\begin{table}[!t]
\caption{Mathematical symbols, variables, operators and parameters.}
\label{tab:variables}
\centering
\renewcommand{\arraystretch}{1.05}
\begin{tabular}{p{1.25cm}<{\centering}|p{6.75cm}<{\centering}}
\toprule[1.5pt]
Symbol & Meaning \\
\midrule
$\lambda$ & modulo threshold \\
$\Omega$ & bandwidth \\
$\rho$ & $\lambda$-exceedance parameter \\
$\varepsilon$ & termination threshold for OMP \\
$\Radon f$ & Radon transform of $f$ \\
$\Mod_\lambda$ & modulo operator with threshold $\lambda$ \\
$\Modulo^\lambda f$ & modulo Radon transform of $f$ \\
$\Fourier_d h$ & $d$-dimensional Fourier transform of $h$ \\
$p_\bftheta$ & Radon projection with direction $\bftheta$ \\
$p^\lambda_\bftheta$ & modulo Radon projection with direction $\bftheta$ \\
$s_\bftheta^\lambda$ & residual function \\
$\E_{N_\Omega,N}$ & set of in-band frequencies \\
$\Delta$ & forward difference operator \\
$\Sum$ & anti-difference operator \\
$K$ & number of radial samples (left of origin) \\
$K'$ & number of radial samples (right of origin) \\
$M$ & number of angular samples \\
$N$ & number of radial samples reduced by one \\
$N_\Omega$ & effective bandwidth \\
$R$ & size of reconstruction grid \\
$p_\bftheta[k]$ & samples of $p_\bftheta$ at $t_{k-K} = (k-K)\T$ \\
$\underline{p}_\bftheta[k]$ & forward differences of $p_\bftheta[k]$ \\
$\widehat{\underline{p}}_\bftheta[n]$ & discrete Fourier transform of $\underline{p}_\bftheta[k]$ \\
$\widetilde{p}_\bftheta[n]$ & reduced discrete Fourier transform of $p_\bftheta[k]$ \\
$\Ell_\lambda$ & set of frequencies in $\widehat{\underline{s}}_\bftheta^\lambda[n]$ \\
\bottomrule[1.5pt]
\end{tabular}
\end{table}

\section{Revisiting Modulo Radon Transform}
\label{sec:MRT}

For convenience, we summarize the notation in Table~\ref{tab:variables}.

\bpara{Forward Model.} The recently introduced modulo Radon transform (MRT) allows for a range reduction of classical Radon measurements.
For a bivariate function $f \in \Lebesgue^1(\R^2)$, the Radon transform $\Radon f$ is defined as
\begin{equation*}
\Radon f(t,\bftheta) = \int_{\bfx \cdot \bftheta = t} f(\bfx) \: \d \bfx
\quad \text{for } (t,\bftheta) \in \R \times \Sphere. 
\end{equation*}
Its composition with the modulo operator $\Mod_\lambda$ yields the {\em modulo Radon transform} $\Modulo^\lambda f$ with
\begin{equation*}
\Modulo^\lambda f (t,\bftheta) = \Mod_\lambda(\Radon f(t,\bftheta))
\quad \text{for } (t,\bftheta) \in \R \times \Sphere
\end{equation*}
for a fixed modulo threshold $\lambda > 0$, first introduced in~\cite{Bhandari2020,Beckmann2020,Beckmann2022}.
Here, the $2\lambda$-centered modulo operation is given by 
\begin{equation*}
\Mod_\lambda(t) = t - 2\lambda \left\lfloor \frac{t+\lambda}{2\lambda} \right\rfloor
\quad \text{for } t \in \R,
\end{equation*} 
where $\lfloor \cdot \rfloor$ denotes the floor function for real numbers.
For the sake of brevity, in the following, we write ${\Radon_\bftheta f = \Radon f(\cdot,\bftheta)}$ and ${\Modulo_\bftheta^\lambda f = \Modulo^\lambda f (\cdot,\bftheta)}$, respectively.
Moreover, $\bftheta \in \Sphere$ can be written as $\bftheta = \bftheta(\varphi) = (\cos(\varphi), \sin(\varphi))^\top$ with $\varphi \in [0,2\pi)$ and due to the symmetry property
\begin{equation*}
\Modulo^\lambda f (t,\bftheta(\varphi+\pi)) = \Modulo^\lambda f (-t,\bftheta(\varphi))
\end{equation*}
it suffices to consider $t \in \R$ and $\varphi \in [0,\pi)$.
The corresponding inverse problem can be formulated as recovering $f$ from $g$, where the function $g$ fulfils the non-linear operator equation
\begin{equation}
\label{eq:inv_prob}
\Modulo^\lambda f = g
\end{equation}
and, typically, only finitely many noisy samples of $g$ are given.

\bpara{Pre-filtering and Sampling Architecture.}
In~\cite{Beckmann2022} it has been shown that the modulo Radon transform is injective on the Bernstein space of bandlimited integrable functions and that any function $f \in \Lebesgue^1(\R^2)$ of bandwidth $\Omega > 0$ is uniquely determined by the semi-discrete samples $\{\Modulo^\lambda f(k \T,\bftheta) \mid k \in \Z, ~ \bftheta \in \Sphere\}$ with radial sampling rate ${\T < \frac{\pi}{\Omega}}$.
To deal with compactly supported target functions with $\supp(f) \subseteq B_1(0)$, we follow the approach in~\cite{Beckmann2022} and perform a pre-filtering step with the ideal low-pass filter to the Radon transform to enforce bandlimitedness.
Subsequently, we fold the filtered Radon data into the symmetric interval $[-\lambda,\lambda]$, leading to the following sampling architecture: 
\begin{enumerate}[label=S\textsubscript{\arabic*})]
\item \label{enum:pre1} Compute the Radon projections $p_\bftheta = \Radon_\bftheta f \ast \Phi_\Omega$ by applying the ideal low-pass filter with $\Fourier_1 \Phi_\Omega = \ind_{[-\Omega,\Omega]}$.
\item \label{enum:pre2} Compute the modulo Radon projections $p_\bftheta^\lambda = \Mod_\lambda(p_\bftheta)$ with modulo threshold $\lambda > 0$.
\end{enumerate}
Note that, due to the Fourier slice theorem, the pre-filtering step~\ref{enum:pre1} also leads to a bandlimited version $f_\Omega$ of $f$ via
\begin{equation*}
\Fourier_2 f_\Omega(\sigma \bftheta) = \Fourier_1(p_\bftheta) (\sigma) \quad \forall(\sigma,\bftheta) \in \R \times \Sphere, 
\end{equation*}
where $\Fourier_d$ denotes the $d$-dimensional Fourier transform, 
\begin{equation*}
\Fourier_d h(\bfomega) = \int_{\R^d} h(\bfx) \, \e^{-\i \bfomega \cdot \bfx} \: \d \bfx
\quad \text{for } \bfomega \in \R^d.
\end{equation*}

For applications, the data needs to be fully discretized.
Therefore, we sample $p_\bftheta^\lambda$ with respect to both the radial and the angular variable to obtain the dataset in parallel beam geometry,
\begin{equation}
\label{eq:discr_data}
\Bigl\{ p_{\bftheta_m}^\lambda(t_k) \Bigm| k = -K,\hdots,K', ~ m = 0,\hdots,M-1\Bigr\},
\end{equation}
with radial sampling rate $\T < \frac{\pi}{\Omega}$ and angular sampling rate $\Delta \varphi = \frac{\pi}{M}$, i.e., $t_k = k \T$ and $\bftheta_m = \bftheta(\varphi_m)$ with $\varphi_m = m \Delta \varphi$.

\section{Fourier Based Modulo Radon Inversion}
\label{sec:theory}

\subsection{Sequential Reconstruction Approach}
\label{subsec:MRTOMP}

We start with explaining a sequential recovery approach that is based on a Fourier domain interpretation of the MRT reconstruction problem and allows for a reformulation of the reconstruction problem as a frequency fitting task.
More precisely, by applying a classical sparse approximation tool, namely the {\em orthogonal matching pursuit} algorithm, in Fourier domain, the Radon projections are recovered from the modulo Radon projections and, in spatial domain, the target function is reconstructed by applying the classical filtered back projection formula.
We remark that this approach has been introduced in~\cite{Beckmann2022a}, but still asks for a mathematical theory that guarantees recovery.
In the following, we will take first steps in this direction.

\subsubsection*{First Step}

Since $p_\bftheta$ is bandlimited with bandwidth $\Omega > 0$ due to the above pre-filtering step, its Fourier series can be approximated by 
\begin{equation*}
p_\bftheta(t) \approx \sum\nolimits_{|n| \leq N_\Omega} c_n(p_\bftheta) \, \exp(\i \omega_0 n t),
\end{equation*}
where $c_n(p_\bftheta)$ denotes the $n$-th Fourier coefficient of $p_\bftheta$ and the effective bandwidth $N_\Omega = \lceil \nicefrac{\Omega}{\omega_0}\rceil$ with $\omega_0 = \frac{2\pi}{(K+K'+1)\T}$ defines the number of frequency samples contained in the supported frequency range $[-\Omega,\Omega]$.
In the following, we require that $K' \geq K \geq N_\Omega + 1$.
For simplicity, we set $N = K + K'$ and write $p_\bftheta[k] = p_\bftheta(t_{k-K})$ for $k = 0,\ldots,N$.
The $\Mod_\lambda$ operation decomposes the Radon projection $p_\bftheta$ into the modulo Radon projection $p_\bftheta^\lambda$ and the step function $s_\bftheta^\lambda$ by 
\begin{equation}
\label{eq:decomp}
p_\bftheta[k] = p_\bftheta^\lambda[k] + s_\bftheta^\lambda[k]
\quad \text{for } k = 0,\ldots,N,
\end{equation}
where we have $s_\bftheta^\lambda[k] = \sum_{\ell} \alpha_\ell \ind_{I_\ell}(k\T)$ with pairwise disjoint proper intervals $I_\ell \subseteq [0,N\T]$ and coefficients ${\alpha_\ell \in 2\lambda \Z}$, cf.~\cite{Beckmann2022, Bhandari:2020:Ja}.
Due to the linearity of the forward difference operator ${\Delta: \R^{N+1} \to \R^N}$, defined by ${\Delta z[k] = z[k+1]-z[k]}$, the modulo decomposition in~\eqref{eq:decomp} also applies to ${\underline{p}_\bftheta = \Delta p_\bftheta}$ and to its discrete Fourier transform $\widehat{\underline{p}}_\bftheta$ given by
\begin{equation*}
\widehat{\underline{p}}_\bftheta[n] = \sum_{k=0}^{N-1} \underline{p}_\bftheta[k] \, \exp\Bigl(-\frac{2\pi \i n k}{N}\Bigr)
\quad \text{for } n = 0,\ldots,N-1.
\end{equation*}
The bandlimitedness of $p_\bftheta$ then implies 
\begin{equation}
\label{eq:decomp2}
\widehat{\underline{p}}_\bftheta^\lambda [n] =
\begin{cases}
\widehat{\underline{p}}_\bftheta[n] - \widehat{\underline{s}}_\bftheta^\lambda[n]  & \text{for }  n \in \E_{N_\Omega,N} \\
- \widehat{\underline{s}}_\bftheta^\lambda[n] & \text{for }  n \notin \E_{N_\Omega,N}
\end{cases}
\end{equation}
with indices ${\E_{N_\Omega,N} = \{0,\ldots,N_\Omega \} \cup \{N-N_\Omega,\ldots,N-1\}}$.
In particular, the signal $\widehat{\underline{s}}_\bftheta^\lambda$ can be written as
\begin{equation*}
\widehat{\underline{s}}_\bftheta^\lambda[n] = \sum\nolimits_{\ell \in \Ell_\lambda} c_\ell \exp\left( -\i \frac{\underline{\omega}_0 n}{\T}t_\ell \right)
\quad \text{for }  n = 0,\dots,N-1
\end{equation*}
with $\underline{\omega}_0 = \frac{2\pi}{N}$, $\Ell_\lambda \subseteq \{0,\dots,N\}$ and $t_\ell \in (\T\Z) \cap [0,N\T]$.

\smallskip

In order to obtain a discretized solution to problem~\eqref{eq:inv_prob}, firstly, we aim at reconstructing $\widehat{\underline{p}}_\bftheta$ and, thus, $\widehat{\underline{s}}_\bftheta^\lambda$ for all $\bftheta \in \Sphere$, which corresponds to the recovery of the set of parameters $\{c_\ell, t_\ell \}_{\ell \in \Ell_\lambda}$.
According to~\eqref{eq:decomp2}, we have access to ${\widehat{\underline{s}}_\bftheta^\lambda[n] = - \widehat{\underline{p}}_\bftheta^\lambda[n]}$ for $n \in \{0,\ldots,N-1\}\setminus \E_{N_\Omega,N} = \E_{N_\Omega,N}^\complement$.
This can be used to formulate the sparse minimization problem 
\begin{equation}
\label{eq:obj}
\min_{\bfc \in \C^{N+1}} \lVert\bfc\rVert_0
\quad \text{ such that } \quad
\bfV \bfc = \bfs,
\end{equation}
where the measurement vector $\bfs \in \C^{N-2N_\Omega-1}$ is given by 
\begin{equation}
\label{eq:def_s}
\bfs_{n-N_\Omega} = \widehat{\underline{s}}_\bftheta^\lambda[n]
\quad \text{for } n \in \E_{N_\Omega,N}^\complement 
\end{equation}
and
\begin{equation}
\label{eq:def_V}
\bfV_{n-N_\Omega,\ell+1} = \e^{-\i \underline{\omega}_0 n \ell}
\quad \text{for } \ell = 0,\ldots,N
\end{equation}
defines a {\em Vandermonde matrix} $\bfV \in \C^{(N-2N_\Omega-1) \times (N+1)}$.
The solution vector $\bfc \in \C^{N+1}$ of~\eqref{eq:obj} encodes the sought parameters $\{c_\ell\}_{\ell_\in \Ell_\lambda}$ as its non-zero vector components and the set $\{t_\ell\}_{\ell \in \Ell_\lambda}$ is determined by the corresponding indices.

We solve the sparse optimization problem by applying the orthogonal matching pursuit~(OMP) algorithm, which was first proposed in~\cite{Mallat1993}.
Here, we use a variant as described in Algorithm~\ref{alg:omp}.
Given the estimates $\{c_\ell,t_\ell\}_{\ell_\in \Ell_\lambda}$, we can form $\widehat{\underline{p}}_{\bftheta} = \widehat{\underline{p}}_{\bftheta}^\lambda + \widehat{\underline{s}}_{\bftheta}^\lambda$ via
\begin{equation*}
\widehat{\underline{s}}_\bftheta^\lambda[n] = \sum\nolimits_{\ell \in \Ell_\lambda} c_\ell \exp\left( -\i \frac{\underline{\omega}_0 n}{\T}t_\ell \right)
\end{equation*}
and compute $\underline{p}_\bftheta = \Delta p_\bftheta$ by its inverse discrete Fourier transform.
To recover the the unfolded Radon projections $p_\bftheta$, we make use of the following notion of {\em compact $\lambda$-exceedance}, first introduced in\cite{Beckmann2022}. 

\begin{definition}[Compact $\lambda$-exceedance]
\label{def:comp_lambda_ex}
For $\lambda > 0$, a function $g: \R \to \R$ is called of {\em compact $\lambda$-exceedance with parameter~$\rho > 0$} if
\begin{equation*}
|g(t)| < \lambda
\quad \mbox{ for }  |t| \ > \rho,
\end{equation*}
in which case we write $g \in \Band_\rho^\lambda$.
\end{definition}
Assuming $p_\bftheta \in \Band_\rho^\lambda$ for all $\bftheta \in \Sphere$ and $K \geq \rho\T^{-1}$, we can recover $p_\bftheta$ by applying the anti-difference operator $\Sum: \R^N \to \R^{N+1}$, defined by $\Sum z[k] = \sum_{j < k} z[j]$, as we then have
\begin{equation*}
p_\bftheta = \Sum(\Delta p_\bftheta) + p_\bftheta[0] = \Sum(\Delta p_\bftheta) + p_\bftheta^\lambda[0].
\end{equation*}

\subsubsection*{Second Step}

Having recovered the Radon projections $p_\bftheta$ from modulo Radon projections $p_\bftheta^\lambda$, we now reconstruct the target function $f$ by applying the approximate filtered back projection (FBP) formula
\begin{equation}
\label{eq:FBP_approximate}
f_\FBP = \frac{1}{2} \Radon^\# (F_\Omega * p_\bftheta),
\end{equation}
where $F_\Omega$ is a bandlimited reconstruction filter satisfying $\Fourier_1 F_\Omega(S) = |S| \, W(\frac{S}{\Omega})$ with an even window $W \in \Lebesgue^\infty(\R)$ supported in $[-1,1]$ and where $\Radon^\#$ denotes the adjoint operator of the Radon transform given by the back projection
\begin{equation*}
\Radon^\# g(\bfx) = \frac{1}{2\pi} \int_{\Sphere} g(\bfx \cdot \bftheta, \bftheta) \: \d \bftheta
\quad \text{for } \bfx \in \R^2.
\end{equation*}
As $\Phi_\Omega, F_\Omega$ have the same bandwidth, \eqref{eq:FBP_approximate} can be rewritten as
\begin{equation*}
f_\FBP = \frac{1}{2} \Radon^\# (F_\Omega * \Radon_\bftheta f)
\end{equation*}
and we refer to~\cite{Beckmann2019, Beckmann2020a, Beckmann2021} for a detailed discussion of the reconstruction error $\|f - f_\FBP\|_{\Lebesgue^2(\R^2)}$.

\setlength{\textfloatsep}{10pt}
\begin{algorithm}[t]
\setstretch{1.1}
\caption{OMP Algorithm}
\label{alg:omp}
\begin{algorithmic}[1]
\Require Signal $\bfs$ and matrix $\bfV$, $\varepsilon > 0$, $\bfc^{(0)} = 0$, $\S^{(0)} = \emptyset$
\medskip
\While{$\lVert\bfV^\ast(\bfs - \bfV \bfc^{(i-1)})\rVert_\infty > \varepsilon $}
\smallskip
\State $j^{(i)} = \argmax_{0 \leq j \leq N} \lvert [\bfV^\ast(\bfs - \bfV \bfc^{(i-1)})]_j \rvert$
\State $\S^{(i)} = \S^{(i-1)} \cup \{j^{(i)}\}$
\State $\bfc^{(i)} = \argmin_{\bfc} \left\{\lVert\bfs - \bfV \bfc\rVert_2 \mid \supp(\bfc) \subseteq \S^{(i)}\right\}$
\smallskip
\EndWhile
\medskip
\Ensure $\{c_\ell,t_\ell \}_{\ell \in \Ell_\lambda}$, where $\Ell_\lambda = \S^{(i_{\text{end}})}$ and $t_\ell = \ell \T$.
\end{algorithmic}
\end{algorithm}

The overall recovery scheme, called \ompfbp method, is summarized in Algorithm~\ref{alg:omp-fbp}, where formula~\eqref{eq:FBP_approximate} is discretized using a standard approach.
In more details, we apply the composite trapezoidal rule to discretize the convolution $\ast$ and back projection $\Radon^\#$ leading to the discrete reconstruction via
\begin{equation*}
f_\FBP(\bfx) = \frac{\T}{2M} \sum_{m=0}^{M-1} \sum_{k = -K}^{K'} F_\Omega(\bfx \cdot \bftheta_m - t_k) \, p_{\bftheta_m}(t_k),
\end{equation*}
and more concisely, $f_\FBP = \frac{1}{2} \Radon^\#_D(F_\Omega *_D p_\bftheta)$.
To reduce the computational costs, one typically evaluates the function
\begin{equation*}
h_{\bftheta_j}(t) = (F_\Omega *_D p_{\bftheta_j})(t)
\quad \mbox{ for } t \in \R
\end{equation*}
only at $t = t_i$, $i \in I$, for a sufficiently large $I \subset \Z$ and interpolates the value $h_{\bftheta_k}(t)$ for $t = \bfx \cdot \bftheta_j$ using a suitable scheme.
Due to~\cite{Natterer2001}, the optimal sampling conditions for fixed bandwidth $\Omega > 0$ read $\T \leq \frac{\pi}{\Omega}$, $K' = K \geq \frac{1}{\T}$ and $M \geq \Omega$.

\subsection{Theoretical Recovery Guarantees}
\label{subsec:rec_guarantee}

Based on Algorithm~\ref{alg:omp-fbp}, we now derive theoretical guarantees for the exact recovery of the Radon samples $p_\bftheta[k]$ from $p_\bftheta^\lambda[k]$ for $k = 0,\ldots,N$ and fixed $\bftheta \in \Sphere$.
To this end, we first prove an injectivity property of certain \emph{reduced Vandermonde matrices}.
Thereon, suitable conditions on the index set $\S$ in Algorithm~\ref{alg:omp} lead to the exact reconstruction of the parameter set $\{c_\ell\}_{\ell \in \Ell_\lambda}$.
Finally, we show that this will imply the exact reconstruction of $p_\bftheta$.

In the following we consider the \emph{Vandermonde matrix}
\begin{equation}
\label{eq:gen_van}
\bfV(\bfz) = \begin{pmatrix}
1 & z_1 & z_1^2 & \cdots & z_1^{L-1} \\
1 & z_2 & z_2^2& \cdots & z_2^{L-1} \\
\vdots & \vdots & \vdots & & \vdots \\
1 & z_J & z_J^2 & \cdots & z_J^{L-1}
\end{pmatrix} \in \C^{J \times L}
\end{equation}
for $\bfz = (z_1,\hdots,z_J)^\top \in \C^J$ with $J \in \N$ and $L \in \N$. 
Furthermore,  we define the \emph{reduced Vandermonde matrix}
\begin{equation*}
\bfV_\S(\bfz) = \bigl(\bfV_m(\bfz)\bigr)_{m \in \S}
\end{equation*}
for $\S \subseteq \{0,\dots,L-1\}$, i.e., $\bfV_\S(\bfz)$ consists of the columns $\bfV_m(\bfz)$ of $\bfV(\bfz)$ with $m \in \S$.
In the following Lemma we show a sufficient condition for the injecticity of certain reduced Vandermonde matrices.

\begin{lemma}
\label{lemma:V_one-to-one}
Let $\bfz = (z_1,\dots,z_J) \in \C^J$ with $z_i \neq z_j$ for all $i \neq j$.
Moreover, let $\S \subseteq \{0,\ldots,L-1\}$.
Then, the reduced Vandermonde matrix $\bfV_\S(\bfz) \in \C^{J \times |\S|}$ of $\bfV(\bfz) \in \C^{J \times L}$ in~\eqref{eq:gen_van} is injective if $\M = \max_{m \in \S} |m| < J$.
\end{lemma}

\begin{proof}
Let the indices of $\S = \{m_j\}_{j=1}^{|\S|}$ be of increasing order, i.e., $m_1 < m_2 < \dots < m_{|\S|}$ so that $m_{|\S|} = \M$.
We denote the columns of $\bfV_\S(\bfz)$ by $\bfV_m = (z_1^{m}, \dots, z_J^{m})^\top \in \C^J$.
Assume that $\gamma_{m_1} \bfV_{m_1} + \gamma_{m_2} \bfV_{m_2} + \dots + \gamma_{\M} \bfV_{\M} = 0$ for coefficients $\gamma_{m_1}, \gamma_{m_2}, \ldots, \gamma_{\M} \in \C$.
For each $1 \leq j \leq J$ we then have
\begin{equation}
\label{eq:lindep}
\gamma_{m_1} z_j^{m_1} + \gamma_{m_2} z_j^{m_2} + \ldots + \gamma_{\M} z_j^{\M} = 0.
\end{equation} 
Therefore,  each $z_j$ must be a root of the polynomial
\begin{equation*}
p(t) = \gamma_{m_1} t^{m_1} + \gamma_{m_2} t^{m_2} + \ldots + \gamma_{\M} t^{\M}.
\end{equation*} 
As~\eqref{eq:lindep} holds for all $j = 1,\dots,J$, $p$ has $J$ distinct roots.
Since $p$ is a complex polynomial of degree at most $\M < J$,  we have $p \equiv 0$ due to the fundamental theorem of algebra.
Therefore, $\gamma_{m_1} = \ldots = \gamma_{\M} = 0$. 
Hence, the column vectors of $\bfV_\S(\bfz)$ are linearly independent.
Consequently, the matrix $\bfV_\S(\bfz)$ has full column rank and, thus, it is injective.
\end{proof}

The result of Lemma~\ref{lemma:V_one-to-one} can be exploited to prove the exact recovery of the coefficients $\{c_\ell,t_\ell\}_{\ell \in \Ell_\lambda}$ by means of Algorithm~\ref{alg:omp}. 

\begin{theorem}(OMP recovery)
\label{thm:omp_rec}
Let $\bfV \in \C^{(N-2N_\Omega-1) \times (N+1)}$ be defined as in~\eqref{eq:def_V} and $\bfs \in \C^{N-2N_\Omega-1}$ be defined as in~\eqref{eq:def_s}.
Then, Algorithm~\ref{alg:omp} recovers the parameters $\{c_\ell,t_\ell \}_{\ell \in \Ell_\lambda}$ in $P$ iterations if $\Ell_\lambda \subseteq \S^{(P)}$ and $\max_{m \in \S^{(P)}} |m| < N-2N_\Omega-1.$
\end{theorem}

\begin{proof}
In our setting, we have $\bfs = \widehat{\underline{s}}_\bftheta^\lambda = \bfV \bfc^\dagger$ with $\bfc^\dagger \in \C^{N+1}$ given by
\begin{equation*}
\bfc^\dagger = \begin{cases} c_\ell & \text{for } \ell \in \Ell_\lambda \\
0 & \text{otherwise.}
\end{cases}
\end{equation*} 
In step 4 of the $P$-th iteration of Algorithm~\ref{alg:omp}, we determine 
\begin{align*}
\bfc^{(P)} &= \argmin_{\bfc \in \C^{N+1}} \{\lVert\bfs - \bfV \bfc\rVert_2 \mid \supp(\bfc) \subseteq \S^{(P)}\} \\
&= \argmin_{\bfc \in \C^{N+1}} \{\lVert\bfV(\bfc^\dagger-\bfc)\rVert_2 \mid \supp(\bfc) \subseteq \S^{(P)}\} \\
&= \argmin_{\substack{\bfc \in \C^{N+1}, \\ \supp(\bfc) \subseteq \S^{(P)} }} \bigl\lVert\bfV_{\S^{(P)}}(\bfc^\dagger\vert_{\S^{(P)}}-\bfc\vert_{\S^{(P)}})\bigr\rVert_2,
\end{align*}
where $\bfc^\dagger \vert_{\S^{(P)}} = (c_j^\dagger)_{j \in \S^{(P)}}$ and $\bfc \vert_{\S^{(P)}} = (c_j)_{j \in \S^{(P)}}$.
Since $\Ell_\lambda \subseteq S^{(P)}$, $\bfc^\dagger$ belongs to the set of valid minimizers and according to Lemma~\ref{lemma:V_one-to-one}, the reduced Vandermonde matrix $\bfV_{\S^{(P)}} $ is injective.
Therefore, we have $\ker (\bfV_{\S^{(P)}}) = \{0\}$, which implies $\bfc\vert_{\S^{(P)}}  = \bfc^\dagger\vert_{\S^{(P)}} $ and, thus, $\bfc^{(P)}  = \bfc^\dagger$.
\end{proof}

\setlength{\textfloatsep}{10pt}
\begin{algorithm}[!t]
\setstretch{1.1}
\caption{\ompfbp Method}
\label{alg:omp-fbp}
\begin{algorithmic}[1]
\Require MRT samples $p_{\bftheta_m}^\lambda[k] = p_{\bftheta_m}^\lambda((k-K)\T)$ for $k \in [0,N]$ and $m \in [0,M-1]$, bandwidth $\Omega > 0$, threshold $\varepsilon > 0$
\medskip
\For{$\bftheta \in \{\bftheta_m\}_{m=0}^{M-1}$}
\smallskip
\State Set $\underline{p}_\bftheta^\lambda[n] = \Delta p_\bftheta^\lambda[n]$ and compute $\widehat{\underline{p}}_\bftheta^\lambda[n]$.
\State Estimate $\{c_\ell, t_\ell \}_{\ell \in \Ell_\lambda}$ by applying Algorithm~\ref{alg:omp}.
\State Compute $\widehat{\underline{s}}_\bftheta^\lambda[n] = \sum_{\ell \in \Ell_\lambda} c_\ell \exp \left( -\i \frac{\omega_0 n}{T} t_\ell \right)$
\State Set $\widehat{\underline{p}}_{\bftheta}[n] = \widehat{\underline{p}}_{\bftheta}^\lambda [n]+ \widehat{\underline{s}}_{\bftheta}^\lambda[n]$ and compute $\underline{p}_\bftheta[k]$.
\State  Estimate $p_\bftheta[k]$ by anti-difference.
\smallskip
\EndFor
\medskip
\Ensure \ompfbp reconstruction $f_\FBP = \frac{1}{2} \Radon^\#_D (F_\Omega *_D p_\bftheta)$.
\end{algorithmic}
\end{algorithm}

We can finally conclude recovery of the unfolded Radon projections $p_\bftheta$ by means of Algorithm~\ref{alg:omp-fbp} under the assumption of compact $\lambda$-exceedance, see Definition~\ref{def:comp_lambda_ex}.

\begin{corollary}
\label{cor:rec_setting}
Let $\lambda > 0$ and, for $\bftheta \in \Sphere$, let $p_\bftheta \in \Band_\rho^\lambda$ be of compact $\lambda$-exceedance with parameter $\rho > 0$.
Moreover, let $p_\bftheta[k] = p_\bftheta((k-K)T)$ for $k = 0,\ldots,N$ with $N = K + K'$. 
Assume that
\begin{equation*}
\T < \frac{\pi}{\Omega}, \quad
K \geq \rho\T^{-1}, \quad
K' \geq \frac{\pi \rho \T^{-1} + (K+1) \Omega \T}{\pi - \Omega \T}.
\end{equation*}
Then, steps 2--6 in Algorithm~\ref{alg:omp-fbp} exactly recover $p_{\bftheta}[k]$ from $p_\bftheta^\lambda[k]$ for $k ={ 0,\dots,N}$ with at most $N-2N_\Omega-1$ iterations in Algorithm~\ref{alg:omp} if $\bfV$ in step~4 of Algorithm~\ref{alg:omp} is replaced by its restriction $\bfV_\S$ to ${\S =\{0,\dots,N-2(N_\Omega-1)\}}$.
\end{corollary}

\begin{proof}
Due to the assumptions $p_\bftheta$ satisfies
\begin{equation*}
|p_\bftheta(t)| \leq \lambda
\quad \forall \, |t| > (K'-2(N_\Omega-1)) \T.
\end{equation*}
Therefore, $p_\bftheta[k] = p_\bftheta^\lambda[k]$ for all $k > N-2(N_\Omega-1)$ and, hence, $\max_{\ell \in \Ell_\lambda} \ell \leq N-2(N_\Omega-1)$.
After at most $P = N-2N_\Omega-1$ iterations we have ${\Ell_\lambda \subseteq \S^{(P)} \subseteq \S}$.
Moreover, by construction we have $\max_{m \in \S^{(p)}} \vert m \vert < N-2N_\Omega-1$.
Thus, Theorem~\ref{thm:omp_rec} implies the exact recovery of $p_\bftheta[k]$ for $k = 0,\ldots,N$.
\end{proof}

Note that by defining the oversampling factor $\OF = \frac{\T_{\text{S}}}{\T}$ with Nyquist rate $\T_{\text{S}} = \frac{\pi}{\Omega}$, the sufficient condition in Corollary~\ref{cor:rec_setting} can be rewritten as
\begin{equation*}
\OF > 1, \quad
K \geq \rho\T^{-1}, \quad
K' \geq \frac{\rho \T^{-1} + (K+1) \OF^{-1}}{1 - \OF^{-1}}.
\end{equation*}
As opposed to this, the standard setting $K' = K$ leads to the condition
\begin{equation*}
\OF > 2
\quad \mbox{ and } \quad
K \geq \frac{\rho \T^{-1} + \OF^{-1}}{1 - 2\OF^{-1}}.
\end{equation*}
Consequently, asymmetric radial samples around the origin allow for a smaller sampling rate at the cost of more samples.

\subsection{Direct Fourier Reconstruction Approach}
\label{subsec:DFR}

The above explained \ompfbp method is a mixture of both a Fourier domain and spatial domain reconstruction approach.
In contrast to this, we now propose a recovery scheme that operates solely in Fourier domain.
The first step of this novel reconstruction algorithm follows the above \ompfbp approach until the domain shift from the frequency domain to the spatial domain is applied.
Here, however, we circumvent the domain shift by utilizing a differentiation property of the discrete Fourier transform, which builds an interface between the approach in~\cite{Beckmann2022a} and the class of direct Fourier inversion methods for the conventional Radon transform, which we now explain in more details.

\subsubsection*{New Second Step}

For $\bfx \in \R^2$, the polar coordinate representation of the inverse Fourier transform is given by
\begin{equation}
\label{eq:2d_finv}
f(\bfx) = \frac{1}{(2\pi)^2}\int_{-\pi}^{\pi} \int_{0}^{\infty}   \sigma \,  \Fourier_2 f(\sigma \bftheta(\varphi)) \, \e^{\i \sigma \bfx \cdot \bftheta(\varphi)} \: \d \sigma \, \d \varphi.
\end{equation}
The class of direct Fourier reconstruction~(DFR) approaches is now characterized by performing the Fourier inversion in~\eqref{eq:2d_finv} while only using the given Radon measurements, substituting
\begin{equation}
\label{eq:FST}
\Fourier_2 f(\sigma\bftheta) = \Fourier_1(\Radon_\bftheta f)(\sigma)
\quad \forall \, (\sigma,\bftheta) \in \R \times \Sphere,
\end{equation} 
which holds according to the Fourier slice theorem. 

This simple Radon inversion strategy is in contrast to the above described FBP method that uses a univariate Fourier inversion step, followed by the application of the back projection operator in spatial domain.
Since the DFR approaches operate in the frequency domain, we can almost directly adapt them to the first step of our \ompfbp reconstruction procedure.
Our goal is to utilize the reconstructed frequency information of the Radon data from the first step for the application of the DFR approach without explicitly computing $p_\bftheta$.
The fusion of both steps results in a Fourier-based reconstruction approach of the original data from modulo Radon samples.

In the first step of the \ompfbp approach we compute $\widehat{\underline{p}}_\bftheta$ and, thus, to apply the DFR approach we need to transition from $\widehat{\underline{p}}_\bftheta$ to $\widehat{p}_\bftheta$.
As the discrete Fourier transform and the forward difference operator are non-commutative, we now formulate and prove a relation between a signal and its forward differences in the Fourier domain.
This relation is captured by the following {\em discrete differentiation property~(DDP)}.

\begin{proposition}(Discrete differentiation property~(DDP))
\label{thm:ddp}
For $z = (z[0],\ldots,z[N]) \in \C^{N+1}$ let $\widetilde{z} \in \C^N$ denote its reduced discrete Fourier transform, for $n = 0,\dots,N-1$ defined as
\begin{equation*}
\widetilde{z}[n] = \sum\nolimits_{k = 0}^{N-1} z[k] \, \e^{\nicefrac{-2\pi \i n k}{N}}.
\end{equation*}
Then, for all $n = 0,\dots,N-1$ we have
\begin{equation*}
\widehat{\underline{z}}[n] = \Bigl(\e^{\nicefrac{2\pi \i n}{N}}-1 \Bigr) \widetilde{z}[n]-\e^{\nicefrac{2\pi \i n}{N}} (z[0]-z[N]).
\end{equation*}
\end{proposition}

\begin{proof}
We use the following direct calculations
\begin{align*}
\widehat{\underline{z}}[n] &= \sum\nolimits_{k=0}^{N-1} \bigl(z[k+1]-z[k]\bigr) \, \e^{\nicefrac{-2 \pi \i n k}{N}} \\
&= \sum\nolimits_{k=0}^{N-1} z[k+1] \, \e^{\nicefrac{-2 \pi \i n k}{N}} - \sum\nolimits_{k=0}^{N-1} z[k] \, \e^{\nicefrac{-2 \pi \i n k}{N}}
\end{align*}
so that
\begin{align*}
\widehat{\underline{z}}[n] &= \e^{\nicefrac{2\pi \i n}{N}} \sum\nolimits_{m=1}^N z[m] \e^{-\nicefrac{2\pi \i n m}{N}} - \widetilde{z}[n] \\
&= \left( \e^{\nicefrac{2 \pi \i n}{N}}-1\right) \widetilde{z}[n] -\e^{\nicefrac{2\pi \i n}{N}} (z[0]- z[N]),
\end{align*}
which completes the proof.
\end{proof}

Applying Proposition~\ref{thm:ddp} to our setting, for $N = K + K'$ and
\begin{equation}
\label{eq:reddft_p}
\widetilde{p}_\bftheta[n] = \sum\nolimits_{k= 0}^{N-1} p_\bftheta [k] \e^{-\i \frac{2\pi n k}{N}}
\quad \text{for }  n = 0,\ldots,N-1
\end{equation}
we have 
\begin{align}
\widehat{\underline{p}}_{\bftheta}[n] &= (\e^{\i \underline{\omega}_0 n}-1) \widetilde{p}_\bftheta[n]- \e^{\i\underline{\omega}_0 n} \left(p_\bftheta[0]- p_\bftheta[N]\right)
\nonumber \\
&= (\e^{\i \underline{\omega}_0  n}-1) \widetilde{p}_\bftheta[n] - \e^{\i\underline{\omega}_0  n} \left(p_\bftheta^\lambda[0]-p_\bftheta^\lambda[N]\right),
\label{eq:ddp_radon1}
\end{align}
where the last equality holds if no folds occur in the first and last sample of the signal $p_\bftheta$.
For compactly supported target functions $f$ this can be ensured if the sampling range is chosen sufficiently large, which can be made precise by considering functions of compact $\lambda$-exceedance, cf. Definition~\ref{def:comp_lambda_ex}.

Note that by rearranging~\eqref{eq:ddp_radon1}, we can determine $\widetilde{p}_\bftheta[n]$ from $\widehat{\underline{p}}_{\bftheta}[n]$ for all $n = 1,\hdots,N-1$.
However, this cannot be applied for $n = 0$ according to the vanishing factor $(\e^{\i \underline{\omega}_0 n}-1)$ in this case.
Alternatively, we apply the {\em mean value property} 
\begin{equation}
\label{eq:mvp1}
\widetilde{p}_\bftheta[0] = \sum\nolimits_{k=0}^{N-1}p_\bftheta[k] = \sum\nolimits_{k = 0}^{N-1} (s_\bftheta^\lambda [k] + p_\bftheta^\lambda[k]),
\end{equation}
where instead of explicitly computing the residual samples $s_\bftheta^\lambda[k]$ for $k = 0,\ldots,N-1$ we exploit the previously obtained parameter set $\{c_\ell\}_{\ell \in \Ell_\lambda}$ from OMP and rewrite~\eqref{eq:mvp1} as
\begin{equation}
\label{eq:mvp2}
\widetilde{p}_\bftheta[0] = \sum\nolimits_{k = 0}^{N-1} \Bigl(\sum\nolimits_{\ell < k} c_\ell + p_\bftheta^\lambda[k]\Bigr).
\end{equation}

\begin{algorithm}[t]
\setstretch{1.1}
\caption{\ompdfr Method}
\label{alg:omp-dfr} 
\begin{algorithmic}[1]
\Require MRT samples $p_{\bftheta_m}^\lambda[k] = p_{\bftheta_m}^\lambda((k-K)\T)$ for $k \in [0,N]$ and $m \in [0,M-1]$, bandwidth $\Omega > 0$, threshold $\varepsilon > 0$
\medskip
\For{$\bftheta \in \{\bftheta_m \}_{m=0}^{M-1}$}
\smallskip
\State Set $\underline{p}_\bftheta^\lambda[n] = \Delta p_\bftheta^\lambda[n]$ and compute $\widehat{\underline{p}}_\bftheta^\lambda[n]$.
\State Estimate $\{c_\ell, t_\ell \}_{\ell \in \Ell_\lambda}$ by applying Algorithm~\ref{alg:omp}.
\State Compute $\widehat{\underline{s}}_\bftheta^\lambda[n] = \sum_{\ell \in \Ell_\lambda} c_\ell \exp \left( -\i \frac{\omega_0 n}{\T} t_\ell \right)$.
\State Set $\widehat{\underline{p}}_{\bftheta}[n] = \widehat{\underline{p}}_{\bftheta}^\lambda [n]+ \widehat{\underline{s}}_{\bftheta}^\lambda [n]$.
\State Apply DDP from~\eqref{eq:ddp_radon1} and~\eqref{eq:mvp2} to obtain $\widetilde{p}_\bftheta[n]$
\State Set $\widehat{f}_\Omega[n\bftheta] = \widetilde{p}_\bftheta[n]$.
\smallskip
\EndFor
\smallskip
\State Compute $f_\DFR$ from $\widehat{f}_\Omega$ using a chosen DFR method.
\medskip
\Ensure \ompdfr reconstruction $f_\DFR$.
\end{algorithmic}
\end{algorithm}

Combining the above described first step and this new second reconstruction step fully circumvents the explicit computation of Radon projections $p_\bftheta$ in spatial domain and, hence, leads to a direct Fourier reconstruction scheme for the modulo Radon transform, called \ompdfr method, which is summarized in Algorithm~\ref{alg:omp-dfr}.
Altogether, the approximate reconstruction $f_\DFR$ is obtained by applying the DDP to the recovered signal $\widehat{\underline{p}}_\bftheta$ from the first step and subsequently using the obtained $\widetilde{p}_\bftheta$ as input data for the selected DFR approach.
The reconstruction scheme uses the modulo Radon projections $p_{\bftheta_m}^\lambda[k]$ for $k = 0,\dots,K+K'$ and $m = 0,\dots,M-1$ at a predefined modulo threshold $\lambda$ and the OMP stopping criterion ${\lVert\bfV^\ast(\bfs - \bfV \bfc^{(i-1)})\rVert \leq \varepsilon}$.
For illustration, the flow diagrams of both the \ompfbp method and the \ompdfr method are depicted in Fig.~\ref{fig:diagrams}.
In the following, we propose a suitable DFR approach for our modulo reconstruction framework. 

\begin{figure}[t]
\centering
\includegraphics[width=1\linewidth]{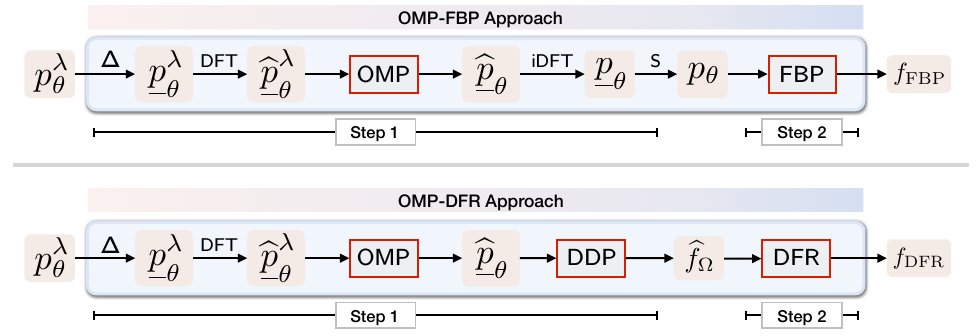}
\caption{Juxtaposition of \ompfbp method (top) and \ompdfr method (bottom).}
\label{fig:diagrams}
\end{figure}

\subsection{Overview of Direct Fourier Reconstruction Methods}
\label{subsec:DFR_Radon}

A major challenge of the DFR approaches lies in the discretization of the two-dimensional Fourier inversion of a function, captured in~\eqref{eq:2d_finv}.
Due to the relation given by~\eqref{eq:FST}, sampling from Radon data with the same strategy as in~\eqref{eq:discr_data} yields the samples of $\Fourier_2 f$ on a polar coordinate grid.
However, to simply apply the common inverse discrete Fourier transform~(IDFT), we require equidistant samples on a Cartesian coordinate grid.
The most intuitive transition between the two grids is a direct interpolation between both coordinate grids.
However, as stated for example in~\cite{Natterer2001}, this may produce severe artifacts in the reconstruction. 

More sophisticated DFR methods improve reconstruction quality by adapting the sampling strategies.
From existing methods, it is possible to distinguish between interpolating and non-interpolating algorithms as well as categorizing them into adapted sampling in the frequency domain and adapted sampling in the spatial domain.

The linogram method, proposed in~\cite{Natterer2001}, is an interpolating method which uses adapted sampling in the frequency domain in order to obtain frequency samples on concentric squares located around the origin, i.e., on a \emph{linogram}.
Another linogram approach from~\cite{Edholm1991} uses sampling points on the linogram in the spatial domain of the Radon transform.
This avoids any kind of interpolation in the frequency domain. 

The gridding method in~\cite{Schomberg1995} uses convolutions with a suitable window function rather than interpolations to approximate the samples of $\Fourier_2 f$ on a Cartesian grid.
A further method is introduced in~\cite{Potts2001}.
It classifies as a non-interpolating DFR method with adapted sampling in Fourier domain and is based on an efficient computation of the inverse discrete Fourier transform by utilizing a suitable algorithm for non-equispaced samples in Fourier domain.

In this work, we focus on the approach proposed in~\cite{Potts2001}, since it is suitable to our framework and has shown reasonable reconstruction quality, as e.g. captured in Section~\ref{sec:numerics}.
In this case, formula~\eqref{eq:2d_finv} is approximated by utilizing non-equispaced discrete Fourier inversion, that is efficiently implemented by the \emph{non-equispaced fast Fourier transform~(NFFT)} algorithm. 
Accordingly, we refer to this as the NFFT-Radon approach and choosing this reconstruction method in Algorithm~\ref{alg:omp-dfr} as DFR method leads to the \ompnfft algorithm.
In the following, we briefly explain the discretization steps in~\cite{Potts2001}.

\smallskip

Let $f_\Omega$ be the pre-filtered version of $f$ with bandwidth~$\Omega$ due to the pre-filtering steps in Section~\ref{sec:MRT}.
We can write the bivariate Fourier inversion formula~\eqref{eq:2d_finv} as
\begin{equation}
\label{eq:2dFT}
f_\Omega(\bfx) = \frac{1}{8\pi^2} \int_{-\pi}^{\pi} \int_{-\infty}^{\infty} \vert \sigma \vert \Fourier_2 f_\Omega(\sigma \bftheta(\varphi)) e^{\i \sigma \bfx \cdot \bftheta(\varphi)} \: \d \sigma \: \d \varphi
\end{equation}
and set 
$h(\sigma) =  \vert \sigma \vert   \Fourier_2 f_\Omega (\sigma \bftheta(\varphi)) e^{\i \sigma(\bfx \cdot \bftheta(\varphi))}$ for $\sigma \in \R$.
Then, Poisson's summation formula yields 
\begin{equation*}
\Fourier_1 h (0) + \sum\nolimits_{n \in \Z \setminus \{0\}} \Fourier_1 h(2n)
= \pi \sum\nolimits_{k = -K}^{K-1} h(\pi k),
\end{equation*} 
where the truncation of the series on the right hand side is valid if the sampling condition $K > \frac{\Omega}{\pi}$ is satisfied.

Assuming that $h$ is essentially bandlimited with sufficiently small bandwidth $\Omega_h \ll 2$, a result from~\cite{Natterer2001a} states that
\begin{equation*}
\sum\nolimits_{n \in \Z \setminus \{0\}} \Fourier_1 h(2 n)
\approx -\frac{\pi^2}{6} \Fourier_2 f_\Omega(0).
\end{equation*}
With this, the previous steps give the approximation 
\begin{equation*}
\begin{split}
& \int_{-\infty}^{\infty} \vert \sigma \vert \Fourier_2 f_\Omega (\sigma \bftheta(\varphi))\e^{\i \sigma (\bfx \cdot \bftheta(\varphi))} \: \d\sigma \\
& \enspace \approx \pi^2 \sum\nolimits_{k = -K}^{K-1} \vert k \vert \Fourier_2 f_\Omega (\pi k \bftheta(\varphi))\e^{\i \pi k(\bfx \cdot \bftheta(\varphi))} + \frac{\pi^2}{6} \Fourier_2 f_\Omega(0)
\end{split}
\end{equation*}
of the inner integral in~\eqref{eq:2dFT}. The outer integral can be discretized by applying the trapezoidal rule with step size $\Delta \varphi = \frac{\pi}{M}$, yielding the approximation
\begin{equation}
\label{eq:disc}
f_\Omega(\bfx) \approx \frac{\pi}{8M} \sum_{m = - M}^{M-1} \sum_{n = -K}^{K-1} \nu_n \Fourier_2 f_\Omega(\pi \bftheta_m n) \e^{\i \pi n \bfx \cdot \bftheta_m}, 
\end{equation} 
where
\begin{equation*}
\nu_n = \begin{cases}
\frac{1}{12} & \text{for } n = 0, \\
n  & \text{otherwise.}
\end{cases}
\end{equation*}

Together with the Fourier slice theorem, the approximation
\begin{equation*}
\Fourier_2 f_\Omega (\pi \bftheta(\varphi_m) n) = \Fourier_1 (p_{\bftheta_m})(\pi n) \approx \T \, \widecheck{p}_{\bftheta_m}[n]
\end{equation*}
for $n \in -K,\hdots,K-1$ is used in~\cite{Potts2001}, where $\widecheck{p}_\bftheta[n]$ is obtained via
\begin{equation*}
\widecheck{p}_\bftheta[n] = \sum\nolimits_{k=-K}^{K-1} p_{\bftheta} (k\T) \, \e^{-\i\pi n k /K}
\end{equation*}
and we set $K' = K$, which can be generalized.

Following the steps of Algorithm~\ref{alg:omp-dfr}, $\widetilde{p}_\bftheta[n]$ is computed as in~\eqref{eq:reddft_p} for $n = 0,\hdots, 2K$ before applying the chosen DFR approach. According to former index conventions, we cannot use $\widecheck{p}_\bftheta$ and $\widetilde{p}_\bftheta$ interchangeably for the computation of~\eqref{eq:disc}, although both versions of the DFT are based on the same set of Radon samples. However, the relation 
\begin{equation*}
\widecheck{p}_\bftheta[n] = (-1)^n \widetilde{p}_\bftheta[n]
\end{equation*}
can be derived, which can be applied in order to compute the Fourier inversion given in~\eqref{eq:disc} within the previous setting of Algorithm~\ref{alg:omp-dfr}. For $n = -K,\hdots,-1$ we use $\widetilde{p}_\bftheta[n] = \widetilde{p}_\bftheta[n+2K]$ to avoid additional computations. 

Since~\eqref{eq:disc} requires a large number of algorithmic operations, the authors of~\cite{Potts2001} provide a suitable algorithm for a more efficient computation.
The proposed \emph{non-equispaced fast Fourier transform~(NFFT)} works for samples that are not necessarily equispaced in time or frequency domain, where~\eqref{eq:disc} corresponds to the latter case. The algorithm is of the same computational complexity as FFT.   

Similar to the FBP method, in~\eqref{eq:disc} we can apply a low-pass reconstruction filter with even window $W \in \Lebesgue^\infty(\R)$ satisfying $\supp(W) \subseteq [-1,1]$ in order to filter the high frequency components, leading to the approximate reconstruction formula
\begin{equation*}
f_\NFFT(\bfx) = \frac{\pi \T}{8M} \sum_{m = - M}^{M-1} \sum_{n = -K}^{K-1} \nu_n W \Big(\frac{\pi n}{N_\Omega}\Big) \widecheck{p}_{\bftheta_m}[n]\e^{\i \pi n \bfx \cdot \bftheta_m}.
\end{equation*}

To close this section, we want to comment on the varying data sizes in Algorithms~\ref{alg:omp-fbp} and~\ref{alg:omp-dfr} during the reconstruction process.
By applying the DDP in step 6 of Algorithm~\ref{alg:omp-dfr}, we can compute $2K$ samples of $\widetilde{p}_\bftheta$, which perfectly suits the setting of the NFFT-based reconstruction scheme.
In contrast to this, we can recover $2K+1$ samples of $p_\bftheta$ in Algorithm~\ref{alg:omp-fbp} by applying by the anti-difference operator, which in turn suits the setting of the FBP reconstruction scheme.

\subsection{Runtime Complexity}
\label{subsec:complexity}

Despite initial drawbacks of direct Fourier inversion methods in terms of reconstruction quality, this approach has been of particular interest in the past due to its beneficial runtime complexity.
Compared to the FBP, the improved complexity is a consequence of embedding the 2D discrete Fourier transform, which can be efficiently implemented by the well-known bivariate fast Fourier transform~(FFT) algorithm.
It has a runtime complexity of $\O(R^2 \log R)$ for signals of size $R \times R$.
Based on the embedding, a majority of direct Fourier inversion methods result in this total order of computational complexity from an asymptotic perspective.
This is in contrast to the filtered back projection algorithm, which has a runtime of $\O(M R^2 + M K \log K)$ with angular sample size $M$, radial sample size $K$ and a quadratic reconstruction grid of size $R \times R$~(cf.~\cite{Natterer2001,Natterer2001a}). 

This advantageous behaviour of DFR methods can be transferred when comparing the \ompfbp method to the class of \ompdfr methods.
Since the latter uses the computationally more efficient Fourier inversion strategy and directly adapts to a majority of former steps from the \ompfbp algorithm, we find an improvement of computational costs for our proposed algorithm.
As proposed in Section~\ref{subsec:DFR}, for demonstration purposes we choose the DFR approach from~\cite{Potts2001} for our further comparison.
Note that other approaches of the same type are provided with comparable computational costs by the design of the direct Fourier inversion in formula~\eqref{eq:2d_finv}.

The first four steps in Algorithm~\ref{alg:omp-fbp} and Algorithm~\ref{alg:omp-dfr} coincide and include the univariate FFT algorithm, the OMP algorithm given in Algorithm~\ref{alg:omp} as well as a sparse matrix multiplication.
The FFT algorithm in step 2 has a computational complexity of $\O(K \log K)$.
Regarding the OMP algorithm, we see from Algorithm~\ref{alg:omp} that the operations of higher computational costs are given by the index choice in step~2 and the least squares minimization in step~4.
In~\cite{Sturm2012} it is shown that the efficiency of these steps can be improved by recursively adapting the residual $\bfs - \bfV \bfc$ and applying a QR-decomposition to the reduced-column version of $\bfV$ at each iteration.
This results in a total complexity of $\O(Ki + i^2)$ for the $i$-th iteration step of the OMP algorithm.
Assuming that the algorithm requires $P$ iteration steps, we arrive at $\O(KP^2 + P^3)$ operations.
The subsequent sparse matrix multiplication has a complexity of $\O(P K)$.
Since all operations are performed for each angle, steps~1--7 of Algorithm~\ref{alg:omp-fbp} have a total complexity of ${\O(M K \log K + MKP^2 + MP^3)}$.
Consequently, the subsequent application of the FBP algorithm leads to an overall complexity of ${\O(M K \log K + MKP^2 + MP^3 + MR^2)}$ for \ompfbp, where we assume that FBP uses an interpolation scheme with computational cost of order $\O(1)$.

As opposed to this, the NFFT-Radon algorithm itself has a computational complexity of $\O(R^2 \log R + R^2)$.
Therefore, the \ompnfft algorithm results in a computational complexity of ${\O(M K \log K + MKP^2 + MP^3 + R^2 \log R + R^2)}$.
Thus, assuming a sufficiently large number of angular samples to achieve a satisfactory reconstruction quality we observe that \ompnfft yields a significant reduction of the total computational complexity compared to \ompfbp.

To explain this further, note that the sampling sizes are often chosen as $M = \O(K)$ and $R = \O(K)$, which is due to the optimal sampling conditions for the Radon transform (cf.~\cite{Natterer2001,Natterer2001a}).
In this setting, the \ompfbp results in a computational complexity of the order ${\O(K^2\log K +K^2P^2+KP^3+K^3)}$ compared to ${\O(K^2\log K +K^2P^2+KP^3+K^2)}$ for the \ompnfft approach. For a fixed number of $P$ iterations in the OMP algorithm, we obtain $\O(K^3)$ as a dominating term for \ompfbp compared to ${\O(K^2\log K)}$ for \ompnfft, which emphasizes the reduction of computational costs by nearly one order of magnitude from an asymptotic perspective.

Finally, note that both algorithms include further computational steps as, for example, the forward difference and anti-difference operator as well as the DDP.
All these operations are of complexity $\O(MK)$ and, therefore, do not act as a dominating term.
For the sake of brevity these have thus been neglected in the considerations of this section.

\section{Numerical and Hardware Experiments}
\label{sec:numerics}

We now present numerical and hardware experiments to demonstrate our approach for both simulated and real data.
To this end, we use the Shepp-Logan phantom~\cite{Shepp1974}, depicted in Fig.~\ref{fig:scheme} (left) along with its sinogram (top right), and the radially symmetric Bull's Eye phantom~\cite{Iske2018} in Fig.~\ref{fig:bullseye}(a), whose sinogram is constant in the angular variable, see Fig.~\ref{fig:bullseye}(b).
We also consider the open source walnut dataset~\cite{Siltanen2015} that includes realistic uncertainties arising from the tomography hardware.

\begin{figure*}[!t]
\centering
\includegraphics[width=\linewidth]{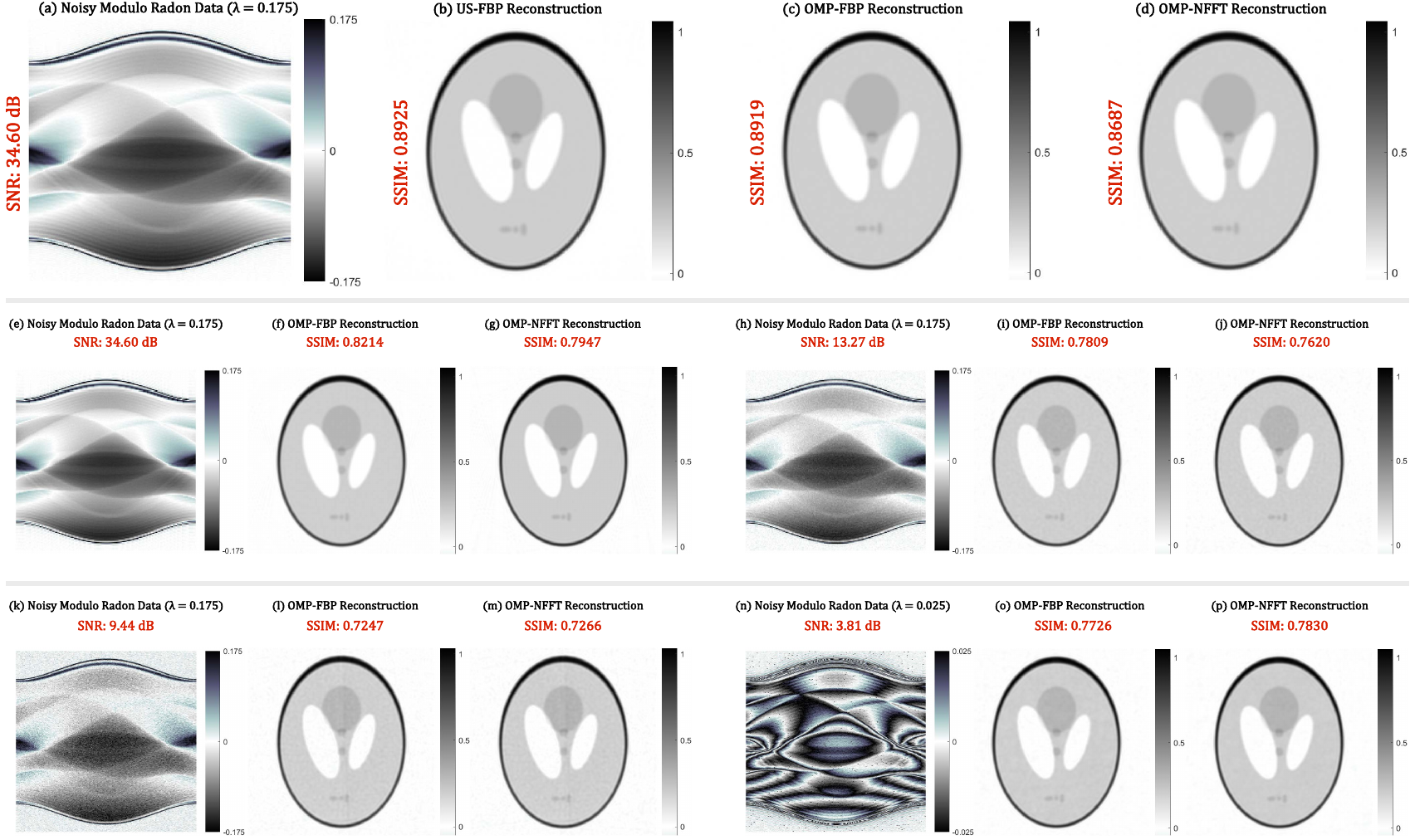}
\caption{Numerical experiments with the Shepp-Logan phantom and various experiment parameters.
(a)~Noisy modulo Radon data with $\lambda = 0.175$ and $\nu = 0.01 \cdot \lambda$. (b)~\usfbp on~(a). (c)~\ompfbp on~(a). (d)~\ompnfft on~(a).
(e)~Noisy modulo Radon data with $\lambda = 0.175$ and $\nu = 0.01 \cdot \lambda$. (f)~\ompfbp on~(e). (g)~\ompnfft on~(e).
(h)~Noisy modulo Radon data with $\lambda = 0.175$, $\sigma = 0.025 \cdot \overline{p}_\bftheta$ and $\nu = 0.025 \cdot \lambda$. (i)~\ompfbp on~(h). (j)~\ompnfft on~(h).
(k)~Noisy modulo Radon data with $\lambda = 0.175$, $\sigma = 0.08 \cdot \overline{p}_\bftheta$ and $\nu = 0.1 \cdot \lambda$. (l)~\ompfbp on~(k). (m)~\ompnfft on~(k).
(n)~Noisy modulo Radon data with $\lambda = 0.025$ and shot noise in $[-0.2,0.2]$. (o)~\ompfbp on~(n). (p)~\ompnfft on~(n).}
\label{fig:experiment_sl}
\end{figure*}

\subsection{Numerical Experiments}

\subsubsection*{Shepp-Logan phantom}
In this section, we present reconstruction results for the Shepp-Logan phantom on a grid of $512 \times 512$ pixels from noisy modulo Radon projections
\begin{equation*}
\{\tilde{p}_{\bftheta_m}^{\lambda}(k\T) \mid -K \leq k \leq K, ~ 0 \leq m \leq M-1\},
\end{equation*}
where we use the optimal parameter choices $\T = \nicefrac{1}{K}$, $\Omega = M$.
Moreover, we set $M = 180$ and use the cosine reconstruction filter with the window function $W(S) = \cos(\frac{\pi S}{2}) \, \ind_{[-1,1]}(S)$.

We start with comparing our proposed OMP-based reconstruction techniques, the \ompfbp and \ompnfft algorithms, with the current state-of-the art method, the \usfbp algorithm from~\cite{Beckmann2022}.
For measuring the reconstruction quality we make use of the structural similarity index measure ($\SSIM$) proposed in~\cite{Bovik2004}.
The results for $\lambda = 0.175$ are summarized in Fig.~\ref{fig:experiment_sl}(a)-(d), where we use the parameter $K = 171$ and uniform noise with noise level $\nu = 0.01 \cdot \lambda$ on the modulo samples, i.e., $\|\tilde{p}_\bftheta^\lambda - p_\bftheta^\lambda\|_\infty \leq \nu$, leading to a signal-to-noise ratio (SNR) of $34.60~\dB$.
This choice leads to an oversampling factor $\OF \approx 3$ and, hence, it violates the recovery condition $\T \leq \frac{1}{\Omega e}$ for \usfbp.
Nevertheless, all three methods succeed and yield a reconstruction of visually the same quality.
More exactly, \usfbp and \ompfbp yield nearly the same $\SSIM$ of $0.89$, whereas \ompnfft leads to a slightly smaller $\SSIM$ of $0.87$.
Note that this is of the same quality as FBP reconstruction from clear Radon data, in which case we have $\SSIM = 0.8957$.

In contrast to this, reducing the oversampling factor to $\OF \approx 1.5$ by choosing $K = 85$ makes \usfbp reconstruction fail.
The OMP-based reconstructions, however, successfully recover with visually the same quality, see Fig.~\ref{fig:experiment_sl}(e)-(g), where we obtain an $\SSIM$ of $0.8214$ for \ompfbp and an $\SSIM$ of $0.7947$ for \ompnfft.
This shows that OMP allows for smaller oversampling as compared to previous approaches and is in conformity with our theoretical findings in Section~\ref{subsec:rec_guarantee}.

\begin{table*}[!t]
\caption{Comparison of MRT reconstruction algorithms in terms of type, recovery condition and runtime complexity.}
\centering
\resizebox{0.75\textwidth}{!}{%
\centering
\renewcommand{\arraystretch}{1.3}
\begin{tabular}{c|cccccc}
\toprule[1.5pt]
\multirow{2}{*}{Method} & \multirow{2}{*}{Type} & Recovery & Runtime & \multicolumn{3}{c}{Latency numbers [sec] for experiment in Fig~\ref{fig:experiment_sl}(a)-(d)} \\
& & condition & complexity & $R=512$  & $R=1024$ & $R=2048$ \\
\midrule
\usfbp & spatial & $\T \leq \frac{1}{\Omega e}$ & $\O(K^3)$ & $1.29$ (FBP: $0.65$) & $2.13$ (FBP: $1.07$) & $4.94$ (FBP: $2.48$) \\
\ompfbp & hybrid & $\T < \frac{\pi}{\Omega}$ & $\O(K^3)$ & $1.24$ (FBP: $0.65$) & $1.67$ (FBP: $1.07$) & $3.06$ (FBP: $2.48$) \\
\ompnfft & frequency & $\T < \frac{\pi}{\Omega}$ & ${\O(K^2\log K)}$ & $0.67$ (NFFT: $0.07$) & $0.76$ (NFFT: $0.14$) & $0.99$ (NFFT: $0.39$) \\[0.5ex]
\bottomrule[1.5pt]
\end{tabular}
}
\label{tab:shepp-logan_phantom}
\end{table*}

\begin{figure}[!t]
\centering
\includegraphics[height=.356\linewidth]{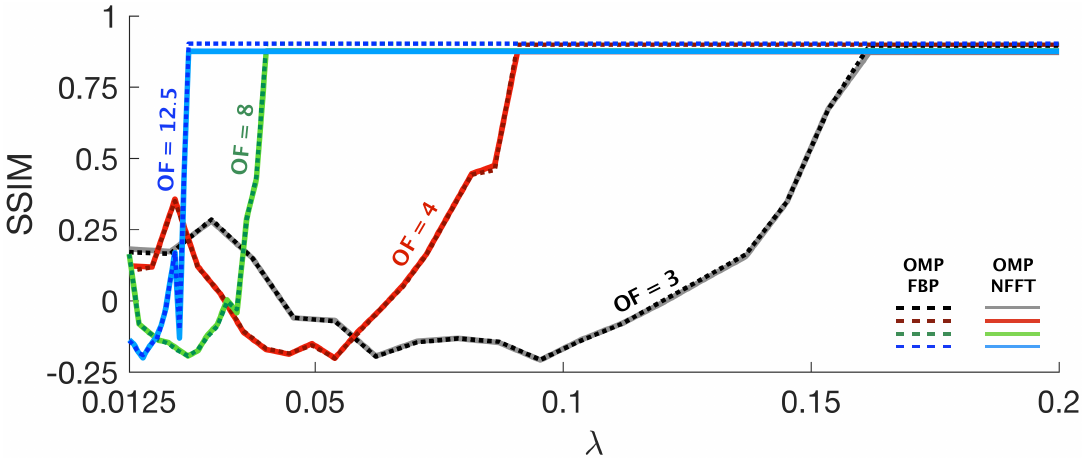}
\caption{Investigation of reconstruction quality for decreasing $\lambda$ and fixed $\OF$.}
\label{fig:lambda_sweep}
\end{figure}

To show the robustness of our approach with respect to noise, in Fig.~\ref{fig:experiment_sl}(h)-(j) we use the parameter $K = 100$, leading to an oversampling factor of $\OF \approx 1.75$, and consider a combination of Gaussian noise before and uniform noise after modulo with threshold $\lambda = 0.175$.
More precisely, we added white Gaussian noise with variance $\sigma^2$ to the Radon projections $p_\bftheta$, where the standard deviation $\sigma = 0.025 \cdot \overline{p}_\bftheta$ depends on the arithmetic mean $\overline{p}_\bftheta$ of $p_\bftheta$.
Moreover, we added uniform noise with noise level $\nu = 0.025 \cdot \lambda$ to the modulo projections, yielding the noisy projections $\tilde{p}_\bftheta^\lambda$ with an SNR of $13.27~\dB$.
While \usfbp fails to recover, we observe that \ompfbp and \ompnfft are able to reconstruct the phantom from this noisy data with again comparable quality.
In numbers, we obtain $\SSIM = 0.7809$ for \ompfbp and $\SSIM = 0.7620$ for \ompnfft.

We observe that the difference in $\SSIM$ between \ompfbp and \ompnfft decreases with an increasing amount of noise.
To demonstrate this, in Fig~\ref{fig:experiment_sl}(k)-(m) we added white Gaussian noise with standard deviation $\sigma = 0.08 \cdot \overline{p}_\bftheta$ to the Radon projections $p_\bftheta$ as well as uniform noise with noise level $\nu = 0.1 \cdot \lambda$ to the modulo Radon projections.
As more noise requires a denser sampling for successful recovery, we use $K = 712$ leading to noisy modulo Radon projections $\tilde{p}_\bftheta^\lambda$ with an SNR of $9.44~\dB$.
In this case, \ompfbp yields an $\SSIM$ of $0.7247$, whereas \ompnfft leads to a slightly larger $\SSIM$ of $0.7266$.

To simulate the observation that real hardware modulo samples can be contaminated with sparse outliers, we consider the case of shot noise in Fig.~\ref{fig:experiment_sl}(n)-(p).
For each direction $\bftheta \in \{\bftheta_j \}_{j=0}^{M-1}$ we added random values in the range $[-0.2,0.2]$ at up to $30$ positions to noisy modulo Radon projections with threshold $\lambda = 0.025$, uniform modulo noise level $\nu = 0.1 \cdot \lambda$ and parameter $K = 821$, leading to the noisy modulo Radon projections $\tilde{p}_\bftheta^\lambda$ with an SNR of $3.81~\dB$.
While \usfbp fails, both OMP-based approaches successfully recover.
Again, \ompnfft yields a larger $\SSIM$ of $0.7830$ as \ompfbp, which leads to an $\SSIM$ of $0.7726$.

Although our OMP-based method is agnostic to $\lambda$, understandably, lower values of $\lambda$ tend to push the numerical methods to their limits. Specifically, in our case, smaller~$\lambda$ invokes higher number of folds worsening the condition number of the reduced Vandermonde matrix. To quantify this effect, we conducted  experiments where we fix oversampling factor, considering noiseless samples. The results are depicted in Fig.~\ref{fig:lambda_sweep}, where we plot the $\SSIM$ of the reconstruction as a function of $\lambda$. While \usfbp succeeds for all choices of $\OF$, we observe that \ompfbp and \ompnfft indeed break down for smaller~$\lambda$. However, the recovery success improves with increasing~$\OF$. The tradeoff between $\lambda$ and $\OF$ remains a compelling research pursuit encouraging more in-depth investigations in the future.

Next, we turn our attention to the noisy scenario considering (i) uniform, (ii) Gaussian, and (iii) mixed noise (combination of (i) and (ii)) while fixing $\OF$ with varying levels of noise. The results are shown in Fig.~\ref{fig:noise_sweep}, where we use $K = 712$ and plot the $\SSIM$ of the reconstruction as a function of SNR (in $\dB$) of the noise, averaged over $10$ runs. In case of uniform noise, we added bounded noise with noise level $\nu \in [0.009,0.2]\lambda$ to the modulo samples, while, in the case of Gaussian noise, we added white Gaussian noise with standard deviation $\sigma \in [0.0012, 0.1]\overline{p}_\bftheta$ to the Radon projections before the modulo operation. In all cases, we observe that \usfbp breaks down for SNR smaller than $20~\dB$, while \ompfbp and \ompnfft still yield decent reconstructions even for noise with very low SNR of $6.5~\dB$. Moreover, \ompnfft seems to perform slightly better in case of more aggressive noise.

\begin{figure}[!t]
\centering
\includegraphics[height=.356\linewidth]{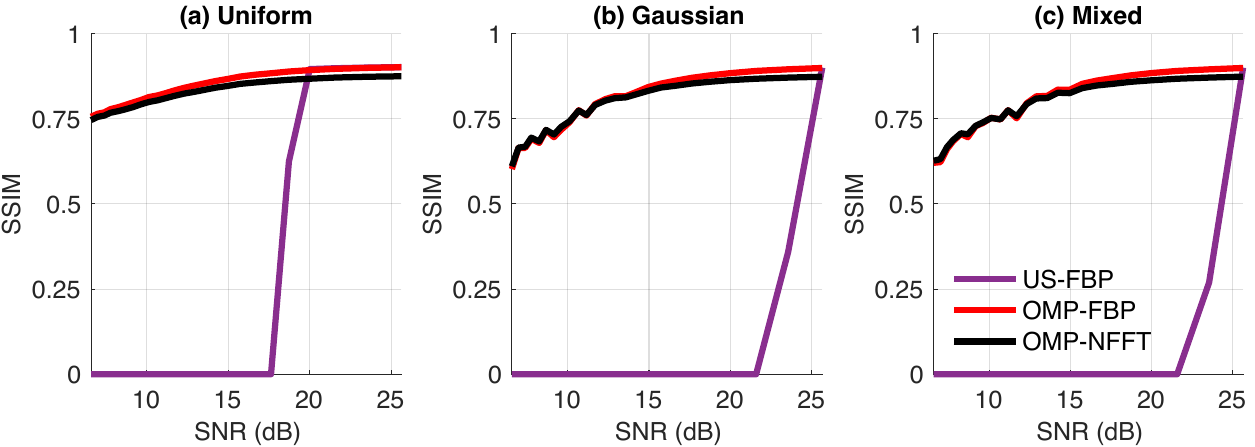}
\caption{Investigation of reconstruction quality for noise with decreasing SNR.}
\label{fig:noise_sweep}
\end{figure}

In conclusion, we observe that our proposed OMP-based reconstruction approaches allow for a smaller oversampling factor compared to the current state-of-the-art \usfbp algorithm, where
\ompfbp and \ompnfft yield comparable reconstruction qualities.
However, \ompnfft is observed to perform better in the present of more severe noise and outperforms \ompfbp in terms of computational complexity, as predicted by the findings in Section~\ref{subsec:complexity} and summarized in Table~\ref{tab:shepp-logan_phantom}.

\begin{figure*}[!t]
\centering
\includegraphics[width=1\linewidth]{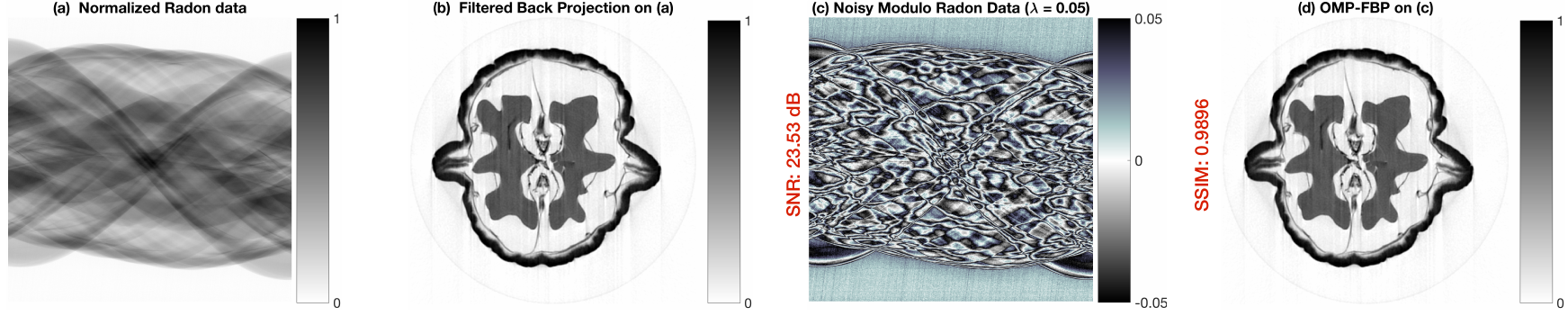}
\caption{Walnut dataset. (a)~Normalized Radon data. (b)~FBP on~(a). (c)~Noisy modulo Radon data with $\lambda = 0.05$ and $\nu = 0.05 \cdot \lambda$. (d)~\ompfbp on~(c).}
\label{fig:walnut_data}
\end{figure*}

\begin{figure}[t]
\centering
\includegraphics[height=.356\linewidth]{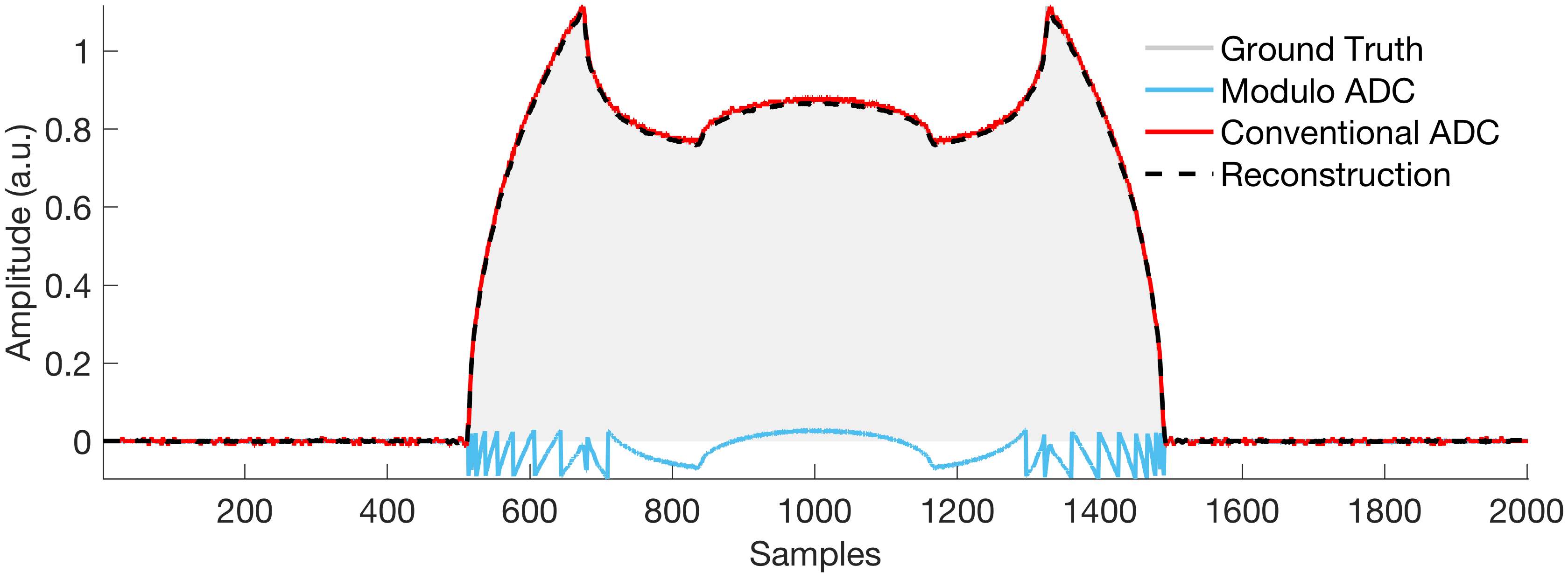}
\caption{Hardware measurements and recovery of Bull's Eye phantom.}
\label{fig:bullseye_line}
\end{figure}

\subsubsection*{Walnut data}
In this section, we present reconstruction results for the walnut dataset from~\cite{Siltanen2015}, which is transformed to parallel beam geometry with $M = 600$ and $K = 1128$ corresponding to $\T = \frac{1}{1128}$.
Moreover, the Radon data is normalized to the dynamical range $[0,1]$, see Fig.~\ref{fig:walnut_data}(a).
The corresponding FBP reconstruction on a grid on $512 \times 512$ pixels is shown in Fig.~\ref{fig:walnut_data}(b), where we used $\Omega = M = 600$ and the cosine reconstruction filter.
Simulated modulo Radon projections with $\lambda = 0.05$ are displayed in Fig.~\ref{fig:walnut_data}(c), where we added uniform noise with noise level $\nu = 0.05 \cdot \lambda$ to the modulo samples to account for quantization errors leading to noisy modulo Radon projections $\tilde{p}_\bftheta^\lambda$ with an SNR of $23.53~\dB$.
The reconstruction with \ompfbp is shown in Fig.~\ref{fig:walnut_data}(d), where we again used $\Omega = M = 600$ and the cosine filter.
We observe that our proposed algorithm yields a reconstruction of the walnut that is indistinguishable from the FBP reconstruction from conventional Radon data, cf.~Fig.~\ref{fig:walnut_data}(b), with an $\SSIM$ of $0.9896$, while compressing the dynamic range by $10$ times.

\subsection{Hardware Validation}
In this section, we describe the pipeline for hardware experiments based on our prototype modulo ADC (\madc); the schematic is shown in \fig{fig:MS}. For a given bit-budget (or digital resolution), we will compare our reconstruction with measurements based on a conventional ADC.

The starting point of our experiment is the Bull's Eye phantom~\cite{Iske2018} in \fig{fig:bullseye}(a), for which we compute the Radon transform analytically. This serves as a numerical benchmark for the hardware data (both via \madc and conventional ADC).
To this end, the Radon transform is sampled and numerically filtered with ideal low-pass filter $\Phi_\Omega$ resulting in $(\Radon_\bftheta f \ast \Phi_\Omega)[n]$, which is subsequently normalized to $[0,1]$. This serves as input to the pipeline depicted in \fig{fig:MS}.
Let us stress that the phantom is compactly supported and, hence, the phantom along with its Radon transform are not bandlimited, as illustrated in \fig{fig:bullseye}(d). Instead, our modified sampling architecture enforces bandlimitedness of the measurements due to pre-filtering with $\Phi_\Omega$.

We then convert the samples to continuous-time representation using a Digital-to-Analog converter (DAC) (\texttt{TTi TG5011}). The output of the DAC is fed to the \madc \cite{Bhandari:2021:J,Beckmann2022} and digitized using an oscilloscope (\texttt{DSO-X 3024A}) resulting in $p_{\bftheta}^\lambda [n]$. Simultaneously, we also record conventional samples mimicking $p_{\bftheta}[n] = (\Radon_\bftheta f \ast \Phi_\Omega)[n]$ via \texttt{DSO-X 3024A}. For both modulo and conventional samples, the measurement resolution is $\sim 6.4$ bits. The respective samples are shown in \fig{fig:bullseye_line}.

Note that the hardware pipeline introduces unknown gains which need to be calibrated and we do so with the knowledge of ground truth data (gray curve in \fig{fig:bullseye_line}). After post-calibration, conventional ADC samples and ground truth match up to ADC resolution ($\sim 6.4$ bits) and we apply steps 2--6 in Algorithm~\ref{alg:omp-fbp} to reconstruct a version of $p_{\bftheta}[n] = (\Radon_\bftheta f \ast \Phi_\Omega)[n]$ given \madc samples (black curve in \fig{fig:bullseye_line}).

\begin{figure}[t]
\centering
\includegraphics[height=.356\linewidth]{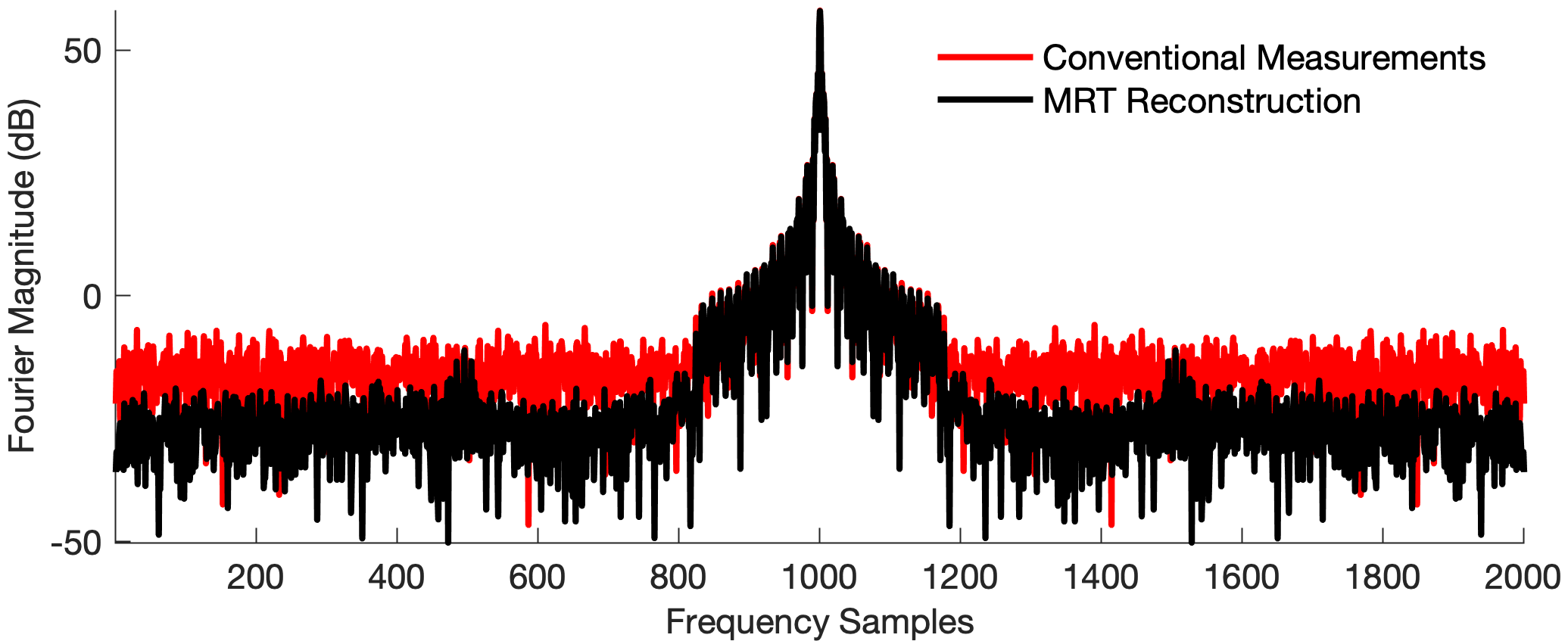}
\caption{Fourier spectrum of data in Fig.~\ref{fig:bullseye_line}.}
\label{fig:FigFT}
\end{figure}

\begin{figure}[!t]
\begin{center}
     \scalebox{0.725}{\tikzset{every picture/.style={line width=0.75pt}} 

\begin{tikzpicture}[x=0.75pt,y=0.75pt,yscale=-1,xscale=1]

\draw   (60,50) -- (140,50) -- (140,90) -- (60,90) -- cycle ;
\draw   (200,60) .. controls (193,60) and (193,51) .. (195,51) .. controls (197,51) and (197,60) .. (190,60) .. controls (183,60) and (183,51) .. (185,51) .. controls (187,51) and (187,60) .. (180,60) .. controls (173,60) and (173,51) .. (175,51) .. controls (177,51) and (177,60) .. (170,60) ;
\draw    (58,70) -- (40,70) -- (40,35) ;
\draw [shift={(60,70)}, rotate = 180] [color={rgb, 255:red, 0; green, 0; blue, 0 }  ][line width=0.75]    (10.93,-3.29) .. controls (6.95,-1.4) and (3.31,-0.3) .. (0,0) .. controls (3.31,0.3) and (6.95,1.4) .. (10.93,3.29)   ;
\draw   (310,50) -- (380,50) -- (380,90) -- (310,90) -- cycle ;
\draw    (170,80) -- (308.33,80) ;
\draw [shift={(310.33,80)}, rotate = 180] [color={rgb, 255:red, 0; green, 0; blue, 0 }  ][line width=0.75]    (10.93,-3.29) .. controls (6.95,-1.4) and (3.31,-0.3) .. (0,0) .. controls (3.31,0.3) and (6.95,1.4) .. (10.93,3.29)   ;
\draw    (210,60) -- (308.33,60) ;
\draw [shift={(310.33,60)}, rotate = 180] [color={rgb, 255:red, 0; green, 0; blue, 0 }  ][line width=0.75]    (10.93,-3.29) .. controls (6.95,-1.4) and (3.31,-0.3) .. (0,0) .. controls (3.31,0.3) and (6.95,1.4) .. (10.93,3.29)   ;
\draw    (140,70) -- (160,70) -- (170,60) ;
\draw    (160,70) -- (170,80) ;
\draw    (200,60) -- (208,60) ;
\draw [shift={(210,60)}, rotate = 180] [color={rgb, 255:red, 0; green, 0; blue, 0 }  ][line width=0.75]    (10.93,-3.29) .. controls (6.95,-1.4) and (3.31,-0.3) .. (0,0) .. controls (3.31,0.3) and (6.95,1.4) .. (10.93,3.29)   ;
\draw    (150,70) -- (150,35) ;
\draw [shift={(150,70)}, rotate = 270] [color={rgb, 255:red, 0; green, 0; blue, 0 }  ][fill={rgb, 255:red, 0; green, 0; blue, 0 }  ][line width=0.75]      (0, 0) circle [x radius= 3.35, y radius= 3.35]   ;
\draw    (290,60) -- (290,35) ;
\draw [shift={(290,60)}, rotate = 270] [color={rgb, 255:red, 0; green, 0; blue, 0 }  ][fill={rgb, 255:red, 0; green, 0; blue, 0 }  ][line width=0.75]      (0, 0) circle [x radius= 3.35, y radius= 3.35]   ;
\draw    (380,60) -- (413,60) ;
\draw [shift={(415,60)}, rotate = 180] [color={rgb, 255:red, 0; green, 0; blue, 0 }  ][line width=0.75]    (10.93,-3.29) .. controls (6.95,-1.4) and (3.31,-0.3) .. (0,0) .. controls (3.31,0.3) and (6.95,1.4) .. (10.93,3.29)   ;
\draw    (380,80) -- (413,80) ;
\draw [shift={(415,80)}, rotate = 180] [color={rgb, 255:red, 0; green, 0; blue, 0 }  ][line width=0.75]    (10.93,-3.29) .. controls (6.95,-1.4) and (3.31,-0.3) .. (0,0) .. controls (3.31,0.3) and (6.95,1.4) .. (10.93,3.29)   ;

\draw (100,70) node  [font=\small] [align=left] {\begin{minipage}[lt]{45.85pt}\setlength\topsep{0pt}
\begin{center}
{\texttt{TTi TG5011}}\\{\texttt{(DAC)}}
\end{center}

\end{minipage}};
\draw (345,70) node  [font=\small] [align=left] {\begin{minipage}[lt]{31.79pt}\setlength\topsep{0pt}
\begin{center}
\sf{DSO-X}\\\sf{3024A}
\end{center}

\end{minipage}};
\draw (153,20.5) node  [font=\large]  {$p_{\bftheta }( t)$};
\draw (437,58.5) node  [font=\normalsize]  {$p_{\bftheta }^{\lambda }[ n]$};
\draw (437,78.5) node  [font=\normalsize]  {$p_{\bftheta }[ n]$};
\draw (59,26) node  [font=\small] [align=left] {\sf{Radon Samples}};
\draw  [draw opacity=0]  (265.5,5.5) -- (315.5,5.5) -- (315.5,40.5) -- (265.5,40.5) -- cycle  ;
\draw (290.5,23) node  [font=\large] [align=left] {$p_{\bftheta }^{\lambda }( t)$};
\draw  [color={rgb, 255:red, 255; green, 0; blue, 0 }  ,draw opacity=1 ][fill={rgb, 255:red, 255; green, 255; blue, 255 }  ,fill opacity=1 ]  (212.38,47) -- (273.38,47) -- (273.38,72) -- (212.38,72) -- cycle (209.38,44) -- (276.38,44) -- (276.38,75) -- (209.38,75) -- cycle ;
\draw (242.88,59.5) node   [align=left] { \ \madc \ };

\end{tikzpicture}}  

\end{center}
\caption{Pipeline for hardware experiments with prototype \madc.}
\label{fig:MS}
\end{figure}
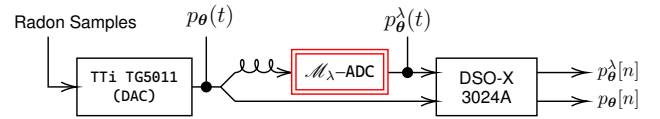

As shown in \fig{fig:FigFT}, {\bf reconstructed data from \madc output yields a lowered quantization noise floor} ($\sim 12~\dB$). This is because for a given bit-budget (here $\sim 6.4$ bits), the MRT measurements enjoy a higher digital resolution, which translates to a higher sensitivity in the MRT samples.

Using the radial symmetry of the Bull's Eye phantom, we can construct its full RT sinogram in \fig{fig:bullseye}(b) as well as its full MRT sinogram in \fig{fig:bullseye}(c) by repeating the conventional ADC samples and the \madc measurements, respectively.
The conventional FBP reconstruction from RT data is shown in \fig{fig:bullseye}(e) and yields an $\SSIM$ of $0.8888$. In contrast to this, \ompfbp based on MRT data leads to an improved $\SSIM$ of $0.9142$, see \fig{fig:bullseye}(f).
The NFFT reconstruction from RT data is shown in \fig{fig:bullseye}(g) with an $\SSIM$ of $0.8741$, whereas \ompnfft leads to $\SSIM = 0.9003$, see \fig{fig:bullseye}(h).

As before, an inspection in Fourier domain, cf. \fig{fig:bullseye}(i)-(l), reveals that the reconstructions from \madc measurements demonstrate higher sensitivity due to lower quantization noise floor resulting in reduced ring artifacts in the reconstructions; \ompnfft shows the best \emph{out-of-band} noise reduction.

\begin{figure*}[t]
\centering
\includegraphics[width=1\linewidth]{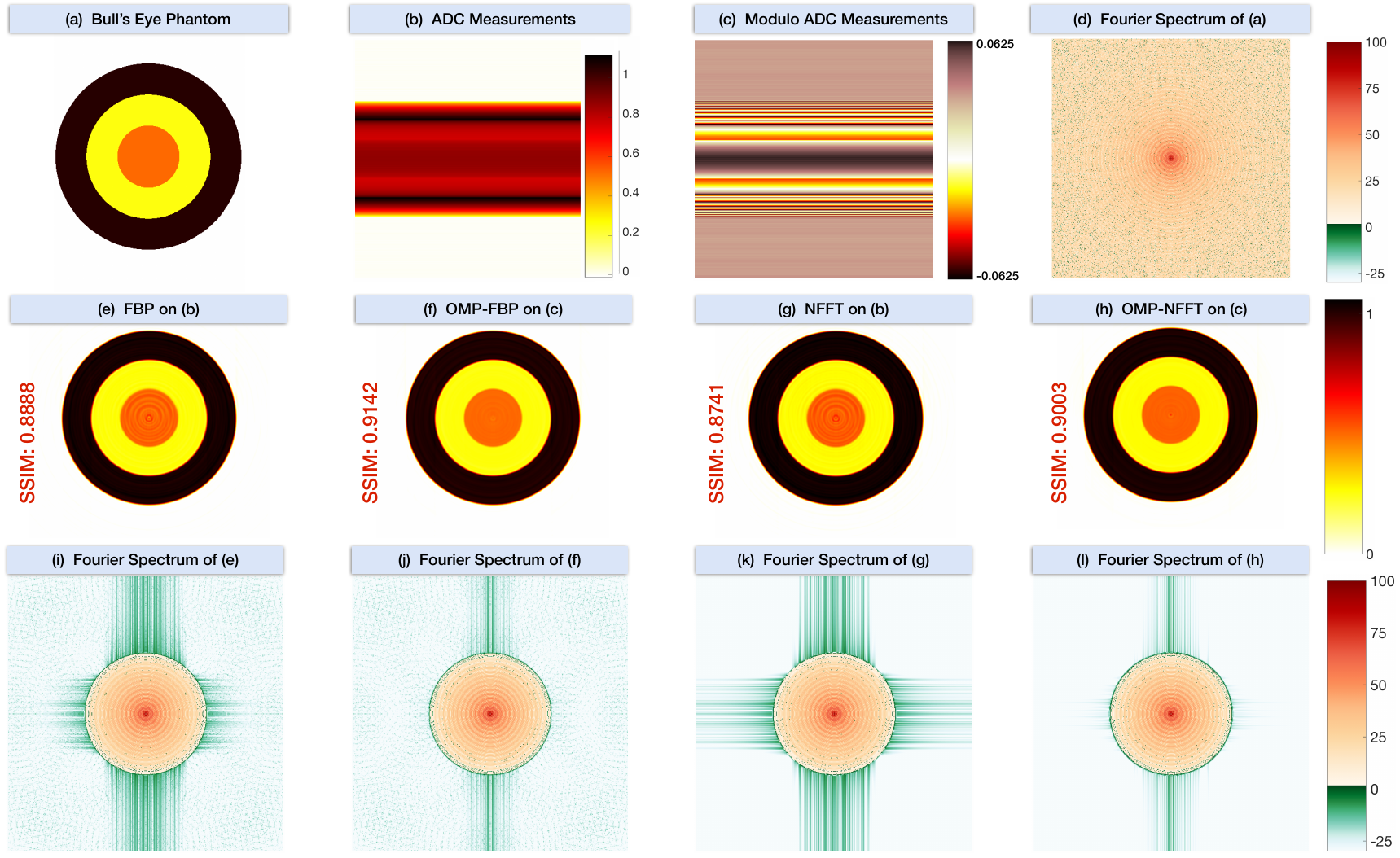}
\caption{Hardware experiment based on Bull's Eye phantom. (a)~Phantom. (b)~ADC measurements. (c)~Modulo ADC measurements. (d)~Spectrum of~(a). (e)~FBP on~(b). (f)~\ompfbp on~(c). (g)~NFFT on~(b). (h)~\ompnfft on~(c). (i)~Spectrum of~(e). (j)~Spectrum of~(f). (k)~Spectrum of~(g). (l)~Spectrum of~(h). All spectra are in $\dB$.}
\label{fig:bullseye}
\end{figure*}

\section{Conclusion}
\label{sec:conc}

The modulo Radon transform (MRT) has been recently introduced as a tool for single shot, high dynamic range tomography. This paper develops a novel recovery method for the inversion of MRT. The proposed approach is backed by mathematical guarantees and offers several advantages over existing art. In particular, it works with near critical sampling rates, it is agnostic to modulo threshold, it is computationally efficient and it is empirically stable to system noise. Both numerical and hardware experiments are used to validate the theoretical claims of this paper.

In terms of future research, key directions that are highly relevant to our setup include, 
\begin{enumerate}[leftmargin = *, label = $\bullet$]
\item theoretically analyzing the noise performance of our proposed algorithms to confirm their empirical stability,
\item leveraging advanced sparse optimization methods that are the essential ingredients of our reconstruction method, and
\item developing reconstruction strategies that can directly work with higher dimensions and tensor based imaging models, which is particularly applicable to voxel-based data.
\end{enumerate}
Beyond the algorithmic aspects, efficient hardware design implementing the MRT remains an interesting topic on its own.

\section*{Acknowledgments}
The authors thank the second reviewer for their very carefully thought through comments and constructive criticism that have greatly enhanced the presentation of our work.

\balance

\end{document}